\documentclass[12pt,letterpaper]{amsart}
\usepackage{graphicx}
\usepackage{amsmath}
\usepackage{amssymb}
\usepackage{amsfonts}
\usepackage{amsthm}
\usepackage{cancel}
\usepackage[all]{xy}
\usepackage[neveradjust]{paralist}
\usepackage{enumitem}
\usepackage{multicol}
\usepackage{mathrsfs}
\usepackage{mathdots}
\usepackage{mathtools}
\usepackage[pagebackref,colorlinks=true,linkcolor=blue,citecolor=blue]{hyperref}
\usepackage{tikz-cd}
\usepackage{tikz}
\usepackage{color}

\tikzcdset{scale cd/.style={every label/.append style={scale=#1},
    cells={nodes={scale=#1}}}}

\newcommand{\R}{\mathbb{R}}

\newcommand{\Jord}{{\mathrm{Jord}}}

\newcommand{\inv}{^{-1}}

\newcommand{\GL}{\mathrm{GL}}
\newcommand{\SO}{\mathrm{SO}}
\newcommand{\OO}{\mathrm{O}}
\newcommand{\SL}{\mathrm{SL}}
\newcommand{\Sp}{\mathrm{Sp}}
\newcommand{\BC}{\mathbb{C}}

\renewcommand{\implies}{\Rightarrow}

\newcommand{\comment}[1]{}

\newtheorem{thm}{Theorem}[section]
\newtheorem{cor}[thm]{Corollary}
\newtheorem{lemma}[thm]{Lemma}
\newtheorem{prop}[thm]{Proposition}
\newtheorem {conj}[thm]{Conjecture}

\newtheorem {ques/conj}[thm]{Question/Conjecture}
\newtheorem {ques}[thm]{Question}
\newtheorem{defn}[thm]{Definition}
\newtheorem{rmk}[thm]{Remark}

\newtheorem{exmp}[thm]{Example}
\newtheorem{algo}[thm]{Algorithm}
\newtheorem{recipe}[thm]{Recipe}

\numberwithin{equation}{section}

\begin{document}

\title[Adams Conjecture and Intersections]{The Adams conjecture and intersections of local Arthur packets}

\author{Alexander Hazeltine}
\address{Department of Mathematics\\
University of Michigan\\
Ann Arbor, MI, 48109, USA}
\email{ahazelti@umich.edu}



\dedicatory{}



\begin{abstract}
    The Adams conjecture states that the local theta correspondence sends a local Arthur packet to another local Arthur packet. M{\oe}glin confirmed the conjecture when lifting to groups of sufficiently high rank and also showed that it fails in low rank. Recently, Baki{\'c} and Hanzer described when the Adams conjecture holds in low rank for a representation in a fixed local Arthur packet. However, a representation may lie in many local Arthur packets and each gives a minimal rank for which the Adams conjecture holds. In this paper, we study the interplay of intersections of local Arthur packets with the Adams conjecture.
\end{abstract}

\maketitle

\tableofcontents

\section{Introduction}

Let $F$ be a non-Archimedean local field of characteristic $0$, $\epsilon\in\{\pm1\},$
$W_n$ be a $-\epsilon$-Hermitian space of dimension $n$ over $F$, and $V_m$ be an $\epsilon$-Hermitian space of dimension $m$ over $F.$  We consider the isometry groups of $W_n$ and $V_m$, which we denote by $G=G_n$ and $H=H_m$, respectively. We suppose further that $(G,H)$ forms a reductive dual pair and fix an additive character $\psi_F$ of $F.$ The local theta correspondence is a map from the set of equivalence classes of irreducible admissible representations, which we denote by $\Pi(G)$, to those of $H$, also denoted by $\Pi(H)$. For $\pi\in\Pi(G)$, we denote its image by $\theta_{W_n,V_m,\psi_F}(\pi).$ 

 We suppose that $G$ is connected hereinafter for simplicity. We refer to \cite{Bor79} for the following definitions. Let $\Gamma_F$ be the absolute Galois group of $F$ and $W_F$ be the Weil group. We also let $\hat{G}(\BC)$ be the complex dual group of $G.$ Later, we will restrict to the case that $G$ is a split symplectic group and  hence $\hat{G}(\BC)$ is an odd special orthogonal group. The group $\Gamma_F$ acts on $\hat{G}(\BC)$ and we let ${}^LG=\hat{G}(\BC)\rtimes \Gamma_F$ be the $L$-group of $G.$ An $L$-parameter of $G$ is a $\hat{G}(\BC)$-conjugacy class of an admissible homomorphism $\phi: W_F\times\SL_2(\BC)\rightarrow {}^L G$. We denote the set of $L$-parameters of $G$ by $\Phi(G).$ The local Langlands correspondence is a finite-to-one map $rec:\Pi(G)\rightarrow\Phi(G)$ satisfying several properties, most notably that certain arithmetic invariants agree. Given $\phi\in\Phi(G)$, the inverse image $\Pi_\phi:=rec^{-1}(\phi)$ is known as the $L$-packet attached to $\phi.$ Given $\pi\in\Pi(G)$, the $L$-parameters $\phi_\pi:=rec(\pi)$ is called the $L$-parameter attached to $\pi.$

For quasi-split symplectic and orthogonal groups, the local Langlands correspondence has been established by Arthur (\cite{Art13}). However, we note that the local Langlands correspondence stated above is stated for connected groups and so minor alterations are required for orthogonal groups. We refer to \cite{AG17b} for a precise treatment for orthogonal groups. When $H$ is a quasi-split even orthogonal group, we write $\Phi(H)$ for the set of $L$-parameters of $H$ and let $\Pi_\phi$ denote the $L$-packet attached to $\phi$ for $\phi\in\Phi(H).$

The idea of Langlands functoriality is the following. Suppose that we have a suitable homomorphism $\gamma:{}^L G \rightarrow {}^L H$. Note that we did not define ${}^LH$ if $H$ is disconnected. As we restrict to the case that $H$ is a quasi-split orthogonal group later, we simply remark that in this setting ${}^LH=H(\BC).$ Given $\phi\in\Phi(G)$, it follows that $\gamma\circ\phi\in\Phi(H).$ Langlands functoriality then roughly predicts that there should be a similar map from $\Pi(G)$ to $\Pi(H)$, say $\pi\mapsto\pi_\gamma$, which preserves $L$-packets, i.e., if $\pi\in\Pi_\phi$, then $\pi_\gamma\in\Pi_{\gamma\circ\phi}.$ 

As the local theta correspondence is a map from $\Pi(G)$ to $\Pi(H)$, Langlands conjectured that there exists a homomorphism $\gamma:{}^L G \rightarrow {}^L H$ which realizes the local theta correspondence as an instance of Langlands functoriality (\cite{Lan75}). However, this turned out to be false. Indeed, examples were later found for which $\pi_1,\pi_2\in\Pi_\phi$ for some $\phi\in\Phi(G)$; however, $\theta_{W_n,V_m,\psi_F}(\pi_2)\not\in\Pi_{\phi_{\theta_{W_n,V_m,\psi_F}(\pi_1)}}.$

Adams proposed a remedy to the failure of Langlands' conjecture (\cite{Ada89}). Namely, that instead of $L$-packets, one should consider certain enlargements known as local Arthur packets. These packets formed a crucial component of Arthur's proof of the local Langlands correspondence for quasi-split symplectic and orthogonal groups (\cite[Theorem 1.5.1]{Art13}). Similar to $L$-packets, local Arthur packets are parameterized by local Arthur parameters.
Roughly, a local Arthur parameter is a direct sum of irreducible representations
\begin{align*}
    &\psi: W_F \times \SL_2(\mathbb{C}) \times \SL_2(\mathbb{C}) \rightarrow {}^L G, \\
    &\psi = \bigoplus_{i=1}^r \phi_i|\cdot|^{x_i} \otimes S_{a_i} \otimes S_{b_i},
\end{align*}
satisfying the following conditions:
\begin{enumerate}
    \item [(1)]$\phi_i(W_F)$ is bounded and consists of semi-simple elements with $\dim(\phi_i)=d_i$;
    \item [(2)] $x_i \in \R$ and $|x_i|<\frac{1}{2}$;
    \item [(3)]the restrictions of $\psi$ to the two copies of $\SL_2(\mathbb{C})$ are analytic.
\end{enumerate}
Here, $S_k$ denotes the unique $k$-dimensional irreducible representation of $\SL_2(\mathbb{C})$. We let $\Psi(G)$ denote the set of local Arthur parameters of $G.$ For $\psi\in\Psi(G)$, Arthur defined the local Arthur packet $\Pi_\psi$ (\cite[Theorem 1.5.1]{Art13}). Ostensibly, Arthur's definition only defined $\Pi_\psi$ as a finite multi-set with elements in $\Pi(G)$; however, M{\oe}glin gave another construction of $\Pi_\psi$ and showed that local Arthur packets are multiplicity-free (\cite{Moe06a, Moe06b, Moe09a, Moe10, Moe11a}). A representation $\pi$ is said to be of Arthur type if $\pi\in\Pi_\psi$ for some local Arthur parameter $\psi.$
We refer to \S\ref{sec Local Arthur packets for symplectic groups} and \S\ref{sec Local Arthur packets for orthogonal groups} for more details on local Arthur parameters and packets in the quasi-split symplectic and even orthogonal cases.

By the Local Langlands Correspondence for $\GL_{d_i}(F)$, the bounded representation $\phi_i$ of $W_F$ can be identified with an irreducible unitary supercuspidal representation $\rho_i$ of $\GL_{d_i}(F)$ (\cite{HT01, Hen00, Sch13}). Consequently, we identify $\psi$ as
\begin{equation*}
  \psi = \bigoplus_{\rho}\left(\bigoplus_{i\in I_\rho} \rho|\cdot|^{x_i} \otimes S_{a_i} \otimes S_{b_i}\right), 
\end{equation*}
where the first sum runs over a finite set of
irreducible unitary supercuspidal representations $\rho$ of $\GL_d(F)$ for $d \in \mathbb{Z}_{\geq 1}$ and $I_\rho$ denotes an indexing set.

For $\psi\in\Psi(G)$, Arthur attached an $L$-parameter $\phi_\psi\in\Phi(G)$ via
\[
\phi_\psi(w,x)=\psi\left(w,x,\left(\begin{matrix}
    |w|^{\frac{1}{2}} & 0 \\
    0 & |w|^{\frac{-1}{2}} 
\end{matrix}\right)\right).
\]
The map $\psi\mapsto\phi_\psi$ gives an injection from $\Psi(G)\rightarrow\Phi(G)$ such that $\Pi_{\phi_\psi}\subseteq\Pi_\psi$ (\cite[Proposition 7.4.1]{Art13}). In this way, $\Pi_\psi$ is an enlargement of the $L$-packet $\Pi_{\phi_\psi}.$

As mentioned earlier, Adams conjectured that the local theta correspondence should preserve local Arthur packets, instead of $L$-packets (\cite{Ada89}). To make this precise, we introduce the following notation. As in \cite[\S3.2]{GI14}, we fix a pair of characters $\chi_W,\chi_V$ associated to $W_n$ and $V_m$ respectively. We also let $\alpha=m-n-1$ be a positive odd integer and $\theta_{W_n, V_m,\psi_F}(\pi)=\theta_{-\alpha}(\pi)$ for $\pi\in\Pi(G)$. Adams conjectured the following.
\begin{conj}[The Adams conjecture, {\cite[Conjecture A]{Ada89}}]\label{conj Adams intro}
    Let $\pi\in\Pi(G_n)$ such that $\pi\in\Pi_\psi$ for some $\psi\in\Psi(G_n).$ If $\theta_{-\alpha}(\pi)\neq 0,$ then $\theta_{-\alpha}(\pi)\in\Pi_{\psi_\alpha}$ where
    \begin{equation}\label{eqn psi_alpha intro}
    \psi_\alpha=(\chi_W\chi_V^{-1}\otimes\psi)\oplus \chi_W\otimes S_1\otimes S_\alpha.
    \end{equation}
\end{conj}

M{\oe}glin showed that if $\alpha\gg 0,$ then Conjecture \ref{conj Adams intro} is true (\cite[Theorem 5.1]{Moe11c}; see also  Theorem \ref{thm Moeglin Adams}). M{\oe}glin also also exhibited examples where Conjecture \ref{conj Adams intro} fails. These failures lead M{\oe}glin to pose several questions \cite[\S6.3]{Moe11c} related to determining when Conjecture \ref{conj Adams intro} holds. These questions were resolved recently by Baki{\'c} and Hanzer (\cite{BH22}) when $G$ is a symplectic group and $H$ is a quasi-split even orthogonal group. Henceforth, we assume that $G$ is a symplectic group and $H$ is a quasi-split even orthogonal group. Suppose that $\pi\in\Pi(G)$ and $\pi\in\Pi_\psi$ for some $\psi\in\Psi(G).$ Using M{\oe}glin's parameterization of local Arthur packets, Baki{\'c} and Hanzer construct representations $\pi_{\alpha}$ (see Recipe \ref{rec pi_alpha} for details) such that for $\alpha\gg 0$, we have $\pi_{\alpha}=\theta_{-\alpha}\in\Pi_{\psi_\alpha}.$ Furthermore if $\pi_{\alpha-2}\neq 0,$ then $\pi_{\alpha-2}=\theta_{-(\alpha-2)}(\pi)\in\Pi_{\psi_{\alpha-2}}$ (\cite[Theorem A]{BH22}; see also Theorem \ref{thm A BH22}). Consequently, we consider
\begin{equation*}
    d(\pi,\psi):=\min\{\alpha_0\geq 0\ | \ \pi_\alpha\neq 0 \ \mathrm{for \ all} \ \alpha\geq\alpha_0 \}.
\end{equation*}
Baki{\' c} and Hanzer used $d(\pi,\psi)$ to answer Conjecture \ref{conj Adams intro} completely.

\begin{thm}[{\cite[Corollary D]{BH22}}]\label{thm a=d intro}
    Suppose that $\pi\in\Pi_\psi$ for some $\psi\in\Psi(G).$ Then Conjecture \ref{conj Adams intro} holds if and only if $\alpha\geq d(\pi,\psi).$
\end{thm}

We note that the construction of $\pi_{\alpha}$ depends on the choice of local Arthur parameter $\psi.$ It is possible that one reason for Conjecture \ref{conj Adams intro} to fail is that $\theta_{-\alpha}(\pi)\not\in\Pi_{\psi_\alpha}$ but instead lies in another local Arthur packet. This failure may be remedied by choosing another local Arthur packet for $\pi.$

Let $\Psi(\pi)=\{\psi\in\Psi(G) \ | \ \pi\in\Pi_\psi\}.$ To each $\psi\in\Psi(\pi)$, we may attach $d(\pi,\psi)$ as above. Since $G$ is a symplectic group, the set $\Psi(\pi)$ can be explicitly computed by work of Atobe (\cite{Ato23}) or, independently, the work of Liu, Lo and the author (\cite{HLL22}). Baki{\'c} and Hanzer suspected that $d(\pi,\psi)$ will be lower if $\psi$ is ``more tempered'' (\cite[p. 5]{BH22}). This idea of ``more tempered'' is made rigorous using the notion of raising operators (see Definition \ref{def raising operator}) that were introduced in \cite{HLL22}. This leads to a partial order on $\Psi(\pi).$ Namely, if $\psi_1,\psi_2\in\Psi(\pi)$ and either $\psi_1=\psi_2$ or $\psi_1=T_1 \circ \cdots \circ T_m(\psi_2)$ for some sequence of raising operators $(T_l)_{l=1}^m$, then we write $\psi_1 \geq_{O} \psi_2$. The main goal of this article is the following theorem which confirms the above suspicion of Baki{\'c} and Hanzer.
\begin{thm}\label{thm main thm intro}
    Let $\pi\in\Pi_\psi$ for some $\psi\in\Psi(G)$. Suppose that $T$ is a raising operator and $\pi\in\Pi_{T(\psi)}.$ Then
    \begin{equation*}
        d(\pi,T(\psi))\leq d(\pi,\psi).
    \end{equation*}
\end{thm}

The partial order $\geq_O$ on $\Psi(\pi)$ remarkably has unique maximal and minimal elements.

\begin{thm}[{\cite[Theorem 1.16(2)]{HLL22}}]\label{thm psi max min intro}
    Suppose that $\pi\in\Pi(G)$ is of Arthur type. Then there exists unique maximal and minimal elements denoted by $\psi^{max}(\pi)$ and $ \psi^{min}(\pi)$, respectively, in $\Psi(\pi)$ with respect to the partial order $\geq_O.$
\end{thm}
Moreover, $\psi^{max}(\pi)$ and $ \psi^{min}(\pi)$ are the unique maximal and minimal elements with represent to many other orderings (\cite[Theorem 1.16]{HLL22} and \cite[Theorems 1.9, 1.12]{HLLZ22}). Furthermore, if $\pi\in\Pi_{\phi_\psi}$ for some $\psi\in\Psi(G)$, then $\psi^{max}(\pi)=\psi$ (\cite[Theorem 9.5]{HLL22}). All of this evidence lead to $\psi^{max}(\pi)$  being called ``the'' local Arthur parameter of $\pi$ in \cite{HLL22}.
Combining Theorems \ref{thm main thm intro} and \ref{thm psi max min intro}, we immediately obtain the following theorem.
\begin{thm}\label{thm main max thm intro}
    Let $\pi$ be a representation of $G$ of Arthur type. Then for any $\psi\in\Psi(\pi)$, we have
    \begin{equation*}
        d(\pi,\psi^{max}(\pi))\leq d(\pi,\psi)\leq d(\pi,\psi^{min}(\pi)).
    \end{equation*}
\end{thm}
In particular, to determine if $\theta_{-\alpha}(\pi)\in\Pi_{\psi_\alpha}$ for some $\psi\in\Psi(G)$, it is sufficient to check it for $\psi=\psi^{min}(\pi).$ On the other hand, if $\theta_{-\alpha}(\pi)\in\Pi_{\psi_\alpha}$ for some $\psi\in\Psi(G)$, then $\theta_{-\alpha}(\pi)\in\Pi_{(\psi^{max}(\pi))_\alpha}$.

The key idea of the proof of Theorem \ref{thm main thm intro} is to use Xu's nonvanishing algorithm (\cite[\S8]{Xu21}; see Algorithm \ref{algo Xu nonvanishing}). This algorithm determines precisely when $\pi_\alpha\neq 0.$ Let $T$ be a raising operator and $\pi\in\Pi_\psi\cap\Pi_{T(\psi)}$ for some $\psi\in\Psi(G).$ Furthermore, we let $\pi_{\alpha}$ be the representation defined using the parameterization of $\pi\in\Pi_\psi$ (see Recipe \ref{rec pi_alpha}). Similarly, we let $\pi_{T,\alpha}$ be the representation defined using the parameterization of $\pi\in\Pi_{T(\psi)}.$ Suppose further that $\pi_{\alpha-2}\neq0$ and $\pi_{T,\alpha}\neq0.$ Using Xu's nonvanishing algorithm, we show that the nonvanishing conditions of $\pi_{\alpha-2}\neq0$ and $\pi_{T,\alpha}\neq0$ imply the nonvanishing conditions for $\pi_{T,\alpha-2}$ and hence $\pi_{T,\alpha-2}\neq 0$.

We remark now on a related problem. First, it would be beneficial to understand how to compute $d(\pi,\psi)$ without using a parameterization of local Arthur packets. In one setting, Baki{\'c} and Hanzer showed that it is (roughly) equal to the first occurrence of the local theta correspondence, but this fails in other settings (\cite[Theorem 2 and Example 7.3]{BH22}, see also Theorem \ref{thm BH up down}). We predict a generalization of this result when $\psi=\psi^{max}(\pi)$ in Conjecture \ref{conj stable Arthur}. Namely, $d(\pi,\psi^{max}(\pi))$ is expected to be the minimal $\alpha_0\geq 1$ such that $\theta_{-\alpha}(\pi)$ is of Arthur type for any $\alpha\geq\alpha_0.$ This conjecture implies that $\theta_{-(\alpha_0-2)}(\pi)$ must not be of Arthur type if $\alpha_0>1$ and hence cannot be controlled by the Adams conjecture (Conjecture \ref{conj Adams intro}).
We further suspect that $d(\pi,\psi)$ is either $d(\pi,\psi^{max}(\pi))$ or the maximum of the obstructions given in Lemmas \ref{lemma partial obstruction dual ui dual zeta=-1}, \ref{lemma partial obstruction dual ui dual 3' zeta=-1}, and \ref{lemma partial obstruction ui inverse 3' zeta=-1}.

Another related problem would be to extend the above results beyond quasi-split symplectic and even orthogonal groups. Indeed, the work of \cite{BH22} is expected to generalize to metaplectic-odd orthogonal and unitary dual pairs. However, the work in this article requires Xu's nonvanishing algorithm (\cite[\S8]{Xu21}; see also Algorithm \ref{algo Xu nonvanishing}) which is proved only for quasi-split symplectic and orthogonal groups and also the results of \cite{HLL22} which are known only for split symplectic and odd special orthogonal groups.

Here is the outline of this article. In \S\ref{sec Setup}, we recall the necessary background and results needed in our study of the Adams conjecture (Conjecture \ref{conj Adams intro}). In particular, in \S\ref{sec Adams conjecture}, we recall the conjecture, along with the relevant results of \cite{BH22, Moe11c} on the topic. We also state our main theorem, Theorem \ref{thm main thm}, and the conjectural description of $d(\pi,\psi^{max}(\pi))$  (Conjecture \ref{conj stable Arthur}) mentioned above. In \S\ref{sec Xu nonvanish}, we recall Xu's nonvanishing algorithm (Algorithm \ref{algo Xu nonvanishing}) along with several relevant results. Finally, in \S\ref{sec proof of main thm}, we describe several obstructions to the Adams conjecture and prove Theorem \ref{thm main thm}. We also provide examples illustrating the main ideas of the paper.

\subsection*{Acknowledgements} 
The author thanks Stephen DeBacker, Baiying Liu, and Chi-Heng Lo for their helpful discussions and constant support and interest.

\section{Setup}\label{sec Setup}

Let $F$ be a non-Archimedean local field of characteristic $0$. We fix $\epsilon=1$ and let
$W_n$ be a $-\epsilon$-Hermitian space of even dimension $n$ over $F$ and $V_m$ be an $\epsilon$-Hermitian space of even dimension $m$ over $F.$ The isometry group of $W_n$, which we denote by $G=G_n$, is a symplectic group and the isometry group of $V_m$, which we denote by $H=H_m$, is an orthogonal group.

Let $\mathbb{H}$ denote a hyperbolic plane. Any $\epsilon$-Hermitian space $V_m$ has a Witt decomposition
\begin{equation}\label{eqn Witt decomp}
    V_m=V_{m_0}+V_{r,r},
\end{equation}
where $m=m_0+2r$, $V_{m_0}$ is anisotropic and $V_{r,r}\cong \mathbb{H}^r.$ The isomorphism class of $V_m$ uniquely determines the Witt index $r$ and the space $V_{m_0}$. Fix an anisotropic $\epsilon$-Hermitian space $V_{m_0}.$ Then we associate a Witt tower to $V_{m_0}$ as follows:
\begin{equation}\label{eqn Witt tower}
    \mathcal{V}=\{V_{m_0}+V_{r,r} \ | \ r\geq 0\}.
\end{equation}

Let $d,c\in F^\times$. We let
$$
V_{(d,c)}=F[X]/(X^2-d)
$$
be the 2-dimensional quadratic space over $F$ with bilinear form
$$
(\alpha,\beta)=c\mathrm{tr}(\alpha\overline{\beta}),
$$
where if $\beta=a+bX,$ then $\overline{\beta}=a-bX.$ We say $V_m$ is associated with $V_{(d,c)}$ if $V_{m_0}\cong V_{(d,c)}.$ 

If $V_m$ is associated with $V_{(d,c)}$, then there exists another orthogonal space $V_m^-$ such that $\mathrm{dim}(V_m)=m=\mathrm{dim}(V_m^-)$ and  $\mathrm{disc}(V_m)=d=\mathrm{disc}(V_m^-)$, but $V_m^-\not\cong V_m.$ If $d\in (F^\times)^2$ and $m>2$ then $V_m^-=D\oplus V_{r-1,r-1}$ where $D$ is the unique quaternion algebra over $F.$ The isometry group of $V_m^-$ is an orthogonal group; however, it is not quasi-split. As Xu's non-vanishing algorithm (Algorithm \ref{algo Xu nonvanishing}) only applies to quasi-split orthogonal groups (\cite{Xu21}), we avoid this case. 

Consequently, hereinafter we assume that $d\not\in (F^\times)^2.$ In this case, $V_m^-$ is associated to $(d,c')$ where $c'\not\in cN_{E/F}(E^\times)$. Here $E=F(\sqrt{d})$ and $N_{E/F}$ is the norm map. Note that this does indeed determine $V_m^-$ since $V_{(d,c)}\cong V_{(d',c')}$ if and only if $d\equiv d' \ \mathrm{mod} (F^\times)^2$ and $c\equiv c' \ \mathrm{mod} N_{E/F}(E^\times)$ where $E=F(\sqrt{d})=F(\sqrt{d'}).$

Now fix $d\not\in (F^\times)^2$ and assume that $V_m$ is associated to $V_{(d,c)}.$ Let $V_m^-$ be the orthogonal space defined above. From $V_m$ and $V_m^-$, we obtain two Witt towers
\begin{align}
\label{eqn tower+} \mathcal{V}^+ &=\{V_{(d,c)}+V_{r,r} \ | \ r\geq 0\},    \\
\label{eqn tower-}  \mathcal{V}^- &=\{V_{(d,c')}+V_{r,r} \ | \ r\geq 0\}.
\end{align}
For any $V\in \mathcal{V}^+\cup \mathcal{V}^-$, the isometry group of $V$ is a quasi-split even orthogonal group $O(V)$.

\subsection{Representations}

Let $G'=G_n,H_m.$ We let $\Pi(G')$ be the set of equivalence classes of irreducible admissible representations of $G'.$ Fix a Borel subgroup $B_{G'}$ of $G'$ and let $P$ be a standard parabolic subgroup. Then $P=MN$ where $M$ is its Levi subgroup and $N$ is its unipotent radical. Furthermore, we have $M\cong\GL_{n_1}(F)\times\GL_{n_2}(F)\times\cdots\GL_{n_r}(F)\times G''$ where $G''$ is a group of the same type as $G'.$ Let $\tau_i$ be an irreducible admissible representation of $\GL_{n_i}(F)$ for $i=1,\dots,r$ and $\sigma$ be an irreducible admissible representation of $G''.$ We let
\[
\tau_1\times\tau_2\times\cdots\times\tau_r\rtimes\sigma=i_P^{G'}(\tau_1\otimes\tau_2\otimes\cdots\otimes\tau_r\otimes\sigma)
\]
denote the normalized parabolic induction from $P$ to $G'.$ We also let $r_{P}(\pi)$ be the Jacquet module of $\pi$ with respect to a (possibly nonstandard) parabolic subgroup $P$ of $G_n.$ In particular, we need it for the opposite parabolic subgroup, denoted $\overline{P}$, to $P.$ We recall an instance of Frobenius reciprocity below.

\begin{lemma}[Frobenius Reciprocity]\label{lemma Frobenius rec}
    Let $P=MN$ and $\pi'$ be  an irreducible admissible representation of $M$ and $\pi$ be an irreducible admissible representation of $G'$. Then we have 
    \begin{enumerate}
        \item $\mathrm{Hom}_{G'}( \pi, i_P^{G'}(\pi') )\cong \mathrm{Hom}_M(  r_{P}(\pi), \pi')$ and
    \item $\mathrm{Hom}_{G'}( i_P^{G'}(\pi'), \pi )
    \cong \mathrm{Hom}_M( \pi', r_{\overline{P}}(\pi)).$
        \end{enumerate}
\end{lemma}

Let $l$ be a positive integer. Similarly, we fix a Borel subgroup of $\GL_{l}(F)$ and let $P'$ be a standard parabolic subgroup. Then $P'=M'N'$ where $M'$ is its Levi subgroup and $N'$ is its unipotent radical. Furthermore, we have $M'\cong\GL_{n_1}(F)\times\GL_{n_2}(F)\times\cdots\times\GL_{n_r}(F)$. Let $\tau_i$ be an irreducible admissible representation of $\GL_{n_i}(F)$ for $i=1,\dots,r$. We let
\[
\tau_1\times\tau_2\times\cdots\times\tau_r
\]
denote the normalized parabolic induction from $P'$ to $\GL_{l}(F).$

Let $\rho$ be an irreducible unitary supercuspidal representation of $\GL_l(F)$ and $x,y\in\mathbb{R}$ such that $y-x\in\mathbb{Z}_{\geq 0}.$ The segment $[x,y]_\rho$ is defined to be the tuple $(\rho|\cdot|^x,\rho|\cdot|^{x+1},\dots,\rho|\cdot|^y).$
The representation
\[
\rho|\cdot|^x\times\rho|\cdot|^{x+1}\times\dots\times\rho|\cdot|^y
\]
has a unique irreducible which we denote by $\Delta_\rho[x,y].$ 

By the Langlands classification for $\GL_l(F)$,  any irreducible representation $\tau$ of $\GL_l(F)$ can be realized as a unique irreducible subrepresentation of a parabolic induction of the form
\[\Delta_{\rho_1}[x_1,y_1]\times\cdots\times\Delta_{\rho_r}[x_r,y_r],\]
where each $\rho_i$ is an irreducible unitary supercuspidal representation of $\GL_{n_i}(F),$ $[x_i,y_i]_{\rho_i}$ is a segment, and $x_1+y_1\leq\cdots\leq x_r+y_r.$ In this setting, we write
$$
\tau=L(\Delta_{\rho_1}[x_1,y_1],\dots,\Delta_{\rho_r}[x_r,y_r]).
$$

We need a specific class of representations known as multisegments. Let $(x_{i,j})_{1\leq i\leq s, 1\leq j \leq t}$ be real numbers such that $x_{i,j}=x_{1,1}-i+j.$ We define a multisegment representation $\sigma$ to be the irreducible representation given by
$$
\sigma=\begin{pmatrix}
x_{1,1} & \cdots & x_{1,t} \\
\vdots & \ddots & \vdots \\
x_{s,1} & \cdots & x_{s,t}
\end{pmatrix}_{\rho}:=L(\Delta_{\rho}[x_{1,1},x_{s,1}],\dots,\Delta_{\rho}[x_{1,t}, x_{s,t}]).
$$
The collection of segments $\Delta_{\rho}[x_{1,1},x_{s,1}],\dots,\Delta_{\rho}[x_{1,t}, x_{s,t}]$ is called a multisegment.
\subsection{Theta Correspondence}

Fix an additive character $\psi_F$ on $F.$ The pair $(G,H)$ is a reductive dual pair of a certain metaplectic group. Hence we consider the Weil representation $\omega_{W_n, V_m, \psi_F}$ of $G\times H.$ For $\pi\in\Pi(G)$, the maximal $\pi$-isotypic quotient of this Weil representation is given by
$$
\pi\boxtimes \Theta_{W_n,V_m,\psi_F}(\pi),
$$
where $\Theta_{W_n,V_m,\psi_F}(\pi)$ is a smooth representation of $H$ called the big theta lift of $\pi.$ We let $\theta_{W_n,V_m,\psi_F}(\pi)$, called the (little) theta lift of $\pi,$ be the maximal semi-simple quotient of $\Theta_{W_n,V_m,\psi_F}(\pi)$. Originally conjectured by Howe (\cite{How79}), the following theorem was first proven by Waldspurger (\cite{Wal90}) when the residual characteristic of $F$ is not 2 and then in full generality by Gan and Takeda (\cite{GT16}) and  Gan and Sun (\cite{GS17}).

\begin{thm}[Howe Duality] Let $\pi_1,\pi_2\in\Pi(G_n).$
\begin{enumerate}
    \item If $\theta_{W_n,V_m,\psi_F}(\pi_2)\neq 0$, then $\theta_{W_n,V_m,\psi_F}(\pi_2)$ is irreducible.
    \item If $\pi_1\not\cong\pi_2$ and both $\theta_{W_n,V_m,\psi_F}(\pi_1)$ and $\theta_{W_n,V_m,\psi_F}(\pi_2)$ are nonzero, then $\theta_{W_n,V_m,\psi_F}(\pi_1)\not\cong \theta_{W_n,V_m,\psi_F}(\pi_2).$
\end{enumerate}
\end{thm}

Oftentimes, it is useful to study theta lifts in towers. Let $\mathcal{V}^+$ and $\mathcal{V}^-$ be as in Equations \eqref{eqn tower+} and \eqref{eqn tower-}.
Write $V^+_r=V_{(d,c)}+V_{r,r}\in\mathcal{V}^+$ and $V^-_r=V_{(d,c')}+V_{r,r}\in\mathcal{V}^-.$ Note that $\mathrm{dim}(V_r^+)=\mathrm{dim}(V_r^+)=2(1+r).$ We define the first occurrence of $\pi$ in $\mathcal{V}^+$ to be the integer
$$
m^+(\pi)=\mathrm{min}\{2(1+r) \ | \ \theta_{W_n,V_r^+,\psi_F}(\pi)\neq 0\}
$$
and similarly we define the first occurrence of $\pi$ in $\mathcal{V}^-$ to be the integer
$$
m^-(\pi)=\mathrm{min}\{2(1+r) \ | \ \theta_{W_n,V_r^-,\psi_F}(\pi)\neq 0\}.$$
The Witt towers $\mathcal{V}^+$ and $\mathcal{V}^-$ satisfy the following conservation relation.

\begin{thm}[Conservation relation, \cite{SZ15}]\label{thm conservation relation}
    Let $\pi\in\Pi(G_n).$ Then
    $$
    m^+(\pi)+m^-(\pi)=2n+4.
    $$
\end{thm}

We assume that $c$ and $c'$ are chosen such that $m^+(\pi)> n+2$ and $m^-(\pi)< n+2.$ Hence we call $\mathcal{V}^+$ the ``going up'' tower and $\mathcal{V}^-$ the ``going down'' tower. 
We remark that it is possible that $m^+(\pi)= n+2=m^-(\pi).$ However, the Adams conjecture holds in full generality in this setting. See Remark \ref{rmk Adams up=down}.

Later, we consider $\alpha=m-n-1$ and write $\theta_{W_n,V_m^\pm,\psi_F}(\pi)=\theta^\pm_{-\alpha}(\pi).$ In this setting, we let $m^{\pm,\alpha}(\pi)=m^\pm(\pi)-n-1.$ When it is clear in context, we suppress the $\pm.$

\subsection{Local Arthur packets for symplectic groups}\label{sec Local Arthur packets for symplectic groups} Recall $G_n=\Sp(W_n)$.  We begin by discussing the theory of local Arthur packets in this case. Note that $\widehat{G}_n(\BC)=\SO_{n+1}(\BC).$
A local Arthur parameter is a direct sum of irreducible representations
$$\psi: W_F \times \SL_2(\mathbb{C}) \times \SL_2(\mathbb{C}) \rightarrow \widehat{G}_n(\BC)$$
\begin{equation}\label{eq decomp psi +}
  \psi = \bigoplus_{i=1}^r \phi_i|\cdot|^{x_i} \otimes S_{a_i} \otimes S_{b_i},  
\end{equation}
satisfying the following conditions:
\begin{enumerate}
    \item [(1)]$\phi_i(W_F)$ is bounded and consists of semi-simple elements, and $\dim(\phi_i)=d_i$;
    \item [(2)] $x_i \in \R$ and $|x_i|<\frac{1}{2}$;
    \item [(3)]the restrictions of $\psi$ to the two copies of $\SL_2(\mathbb{C})$ are analytic, $S_k$ is the $k$-dimensional irreducible representation of $\SL_2(\mathbb{C})$, and 
    $$\sum_{i=1}^r d_ia_ib_i = n+1.
$$ 
\end{enumerate}

Two local Arthur parameters are equivalent if they are conjugate under $\widehat{G}_n(\BC)$. We do not distinguish $\psi$ and its equivalence class in the rest of the paper. We let $\Psi^{+}(G_n)$ denote the equivalence class of local Arthur parameters, and $\Psi(G_n)$ be the subset of $\Psi^+(G_n)$ consisting of local Arthur parameters $\psi$ whose restriction to $W_F$ is bounded. In other words, $\psi$ is in $\Psi(G_n)$ if and only if $x_i=0$ for $i=1,\dots, r$ in the decomposition \eqref{eq decomp psi +}.

By the Local Langlands Correspondence for $\GL_{d_i}(F)$, the bounded representation $\phi_i$ of $W_F$ can be identified with an irreducible unitary supercuspidal representation $\rho_i$ of $\GL_{d_i}(F)$ (\cite{HT01, Hen00, Sch13}). Consequently, we may write
\begin{equation}\label{A-param decomp}
  \psi = \bigoplus_{\rho}\left(\bigoplus_{i\in I_\rho} \rho|\cdot|^{x_i} \otimes S_{a_i} \otimes S_{b_i}\right),  
\end{equation}
where first sum runs over 
irreducible unitary supercuspidal representations $\rho$ of $\GL_d(F)$, $d \in \mathbb{Z}_{\geq 1}$.

Let $\psi$ be a local Arthur parameter as in \eqref{A-param decomp}, we say that $\psi$ is of \emph{good parity} if $\psi \in \Psi(G_n)$, i.e., $x_i=0$ for all $i$, and every summand $\rho \otimes S_{a_i} \otimes S_{b_i}$ is self-dual and symplectic.
We let $\Psi_{gp}(G_n)$ denote the subset of $\Psi(G_n)$ consisting of local Arthur parameters of good parity.

Let $\psi \in \Psi^{+}(G_n).$ From the decomposition \eqref{A-param decomp}, we define a subrepresentation $\psi_{nu,>0}$ of $\psi$ by
\begin{equation}\label{eqn psi_nu,>0}
    \psi_{nu,>0}:= \bigoplus_{\rho}\left(\bigoplus_{\substack{i\in I_\rho,\\ x_i>0}} \rho|\cdot|^{x_i} \otimes S_{a_i} \otimes S_{b_i}\right).
\end{equation}
Since the image of $\psi$ is contained in  $\widehat{G}_n(\BC)$, $\psi$ is self-dual, and hence $\psi$ also contains $(\psi_{nu,>0})^{\vee}$. We define $\psi_{u} \in \Psi(G_{n'})$ for some $n'\leq n$ by 
\begin{align}\label{eq def of psi_u}
    \psi= \psi_{nu,>0} \oplus \psi_u \oplus (\psi_{nu,>0})^{\vee}.
\end{align}
Equivalently,
\[ \psi_{u}:= \bigoplus_{\rho}\left(\bigoplus_{\substack{i\in I_\rho,\\ x_i=0}} \rho \otimes S_{a_i} \otimes S_{b_i}\right). \]

In \cite{Art13}, for a local Arthur parameter $\psi \in \Psi(G_n)$, Arthur constructed a finite multi-set $\Pi_\psi$ consisting of irreducible unitary representations of $G_n.$ We call $\Pi_\psi$ the \emph{local Arthur packet} of $\psi.$ M{\oe}glin showed that $\Pi_\psi$ is multiplicity-free (\cite{Moe11a}). For $\psi \in \Psi^+(G_n)$, Arthur defined (\cite[(1.5.1)]{Art13})
\begin{align}\label{eq def packet +}
    \Pi_{\psi}:= \{ \tau_{\psi_{nu,>0}} \rtimes \pi_u \ | \ \pi_{u} \in \Pi_{\psi_u}   \},
\end{align}
where $\tau_{\psi_{nu,>0}}$ is the following irreducible representation of a general linear group defined over $F$
\begin{equation}\label{eqn tau_{psi_{nu,>0}}}
    \tau_{\psi_{nu,>0}}=\bigtimes_\rho\bigtimes_{i\in I_\rho}\begin{pmatrix}
\frac{a_i-b_i}{2}+x_i & \cdots & \frac{a_i+b_i}{2}-1+x_i \\
\vdots & \ddots & \vdots \\
\frac{-a_i-b_i}{2}+1+x_i & \cdots & \frac{b_i-a_i}{2}+x_i
\end{pmatrix}_{\rho}.
\end{equation}
Since $|x_i|<\frac{1}{2}$ in the decomposition \eqref{A-param decomp}, the parabolic induction in \eqref{eq def packet +} is always irreducible (\cite[Proposition 5.1]{Moe11b}, see Theorem \ref{thm red from nu to gp symp} below). We say that an irreducible representation $\pi$ of $G_n$ is \emph{of Arthur type} if $\pi\in\Pi_\psi$ for some local Arthur parameter $\psi \in \Psi^+(G_n)$.

We further decompose $\psi_{u}$. Suppose $\rho \otimes S_a \otimes S_b$ is an irreducible summand of $\psi_u$ that is either not self-dual, or self-dual but not of the same type as $\psi$. Then $\psi$ must contain the other summand $(\rho\otimes S_a\otimes S_b)^{\vee}=\rho^{\vee} \otimes S_a \otimes S_b$. Therefore, we may choose a subrepresentation $\psi_{np}$ of $\psi_u$ such that
\begin{align}\label{eq decomp of psi_u}
     \psi_{u}= \psi_{np} \oplus \psi_{gp} \oplus \psi_{np}^{\vee},
\end{align}
where $\psi_{gp} $ is of good parity and any irreducible summand of $\psi_{np}$ is either not self-dual or self-dual but not of the same type as $\psi$. In \cite{Moe06a}, M{\oe}glin constructed the local Arthur packet $\Pi_{\psi_u}$ from $\Pi_{\psi_{gp}}$, which we record below.

\begin{thm}[{\cite[Theorem 6]{Moe06a}, \cite[Proposition 8.11]{Xu17}}]\label{thm reduction to gp}
Let $\psi_u \in \Psi(G_n)$ with a choice of decomposition \eqref{eq decomp of psi_u}. Write
\[ \psi_{np}= \bigoplus_{\rho} \left( \bigoplus_{i \in I_{\rho}} \rho \otimes S_{a_i} \otimes S_{b_i}\right),  \]
and consider the following irreducible parabolic induction
\begin{equation}\label{eqn tau_{psi_{np}}}
    \tau_{\psi_{np}}=\bigtimes_\rho\bigtimes_{i\in I_\rho}\begin{pmatrix}
\frac{a_i-b_i}{2} & \cdots & \frac{a_i+b_i}{2}-1 \\
\vdots & \ddots & \vdots \\
\frac{-a_i-b_i}{2}+1 & \cdots & \frac{b_i-a_i}{2}
\end{pmatrix}_{\rho}.
\end{equation}
Then for any $\pi_{gp}\in\Pi_{\psi_{gp}}$ the induced representation $\tau_{\psi_{np}}\rtimes\pi_{gp}$ is irreducible, independent of choice of $\psi_{np}$. Moreover,
$$
\Pi_\psi=\{\tau_{\psi_{np}}\rtimes\pi_{gp} \, | \, \pi\in\Pi_{\psi_{gp}}\}.
$$
\end{thm}

Combined with \eqref{eq def packet +}, we obtain the following.

\begin{thm}[{\cite[Proposition 5.1]{Moe11b}}]\label{thm red from nu to gp symp}
Let $\psi\in\Psi^+(G_n)$ with decomposition 
\begin{equation}\label{eqn psi decomp general}
\psi=\psi_{nu,>0}+\psi_{np}+\psi_{gp}+\psi_{np}^\vee+\psi_{nu,>0}^\vee
\end{equation} 
as above. Then, for any $\pi_{gp}\in\Pi_{\psi_{gp}},$ the induction $\tau_{\psi_{nu,>0}}\times\tau_{\psi_{np}}\rtimes\pi_{gp}$ is irreducible. As a consequence, \begin{equation}\label{non-unitary A-packet}
    \Pi_\psi=\{\tau_{\psi_{nu,>0}}\times\tau_{\psi_{np}}\rtimes\pi_{gp} , | , \pi_{gp}\in\Pi_{\psi_{gp}}\}.
\end{equation}
\end{thm}

\subsection{Local Arthur packets for orthogonal groups} \label{sec Local Arthur packets for orthogonal groups}
Recall $H_m=\OO(V_m)$.  We discuss the theory of local Arthur packets in this case. Note that $\widehat{H}_m(\BC)=O_{m}(\BC).$ Since $H_m$ is quasi-split and not connected, its theory of local Arthur packets requires more finesse.

Let $H_m^\circ$ be the connected component of $H_m.$ That is, $H_m^\circ=\SO(V_m)$ is a quasi-split even special orthogonal group. We have $\widehat{H_m^\circ}=\SO_{m}(\BC).$ A local Arthur parameter for $H_m^\circ$ is a direct sum of irreducible representations
$$\psi: W_F \times \SL_2(\mathbb{C}) \times \SL_2(\mathbb{C}) \rightarrow {}^LH_m^\circ$$
\begin{equation}\label{eq decomp psi + orthog}
  \psi = \bigoplus_{i=1}^r \phi_i|\cdot|^{x_i} \otimes S_{a_i} \otimes S_{b_i},  
\end{equation}
satisfying the following conditions:
\begin{enumerate}
    \item [(1)]$\phi_i(W_F)$ is bounded and consists of semi-simple elements, and $\dim(\phi_i)=d_i$;
    \item [(2)] $x_i \in \R$ and $|x_i|<\frac{1}{2}$;
    \item [(3)]the restrictions of $\psi$ to the two copies of $\SL_2(\mathbb{C})$ are analytic, $S_k$ is the $k$-dimensional irreducible representation of $\SL_2(\mathbb{C})$, and 
    $$\sum_{i=1}^r d_ia_ib_i = m.
$$ 
\end{enumerate}

Two local Arthur parameters are equivalent if they are conjugate under $\widehat{H_m^\circ}(\BC)$ and again we do not distinguish $\psi$ and its equivalence class. We let $\Psi^{+}(H_m^\circ)$ denote the equivalence class of local Arthur parameter, and $\Psi(H_m^\circ)$ be the subset of $\Psi^+(H_m^\circ)$ consisting of local Arthur parameters $\psi$ whose restriction to $W_F$ is bounded.

Fix $c\in H_m \setminus H_m^\circ$ and let $\sigma_0$ be the outer automorphism on $H_m^\circ$ given by conjugation by $c.$ We also set $\Sigma_0$ to be the group generated by $\sigma_0.$ We have $H_m=H_m^\circ\rtimes \Sigma_0.$ Through the dual automorphism $\hat{\sigma}_0,$ we have an action of $\Sigma_0$ on $\Psi^{+}(H_m^\circ).$ 
We let $\Psi^+(H_m)$ and $\Psi(H_m)$ be the sets of $\Sigma_0$-orbits of $\Psi^+(H_m)$ and $\Psi(H_m^\circ),$ respectively. We also let $\Pi^{\Sigma_0}(H_m^\circ)$ denote the set of $\Sigma_0$-orbits of $\Pi(H_m^\circ)$. For $\psi\in\Psi^+(H_m),$ Arthur showed there exists a local Arthur packet $\Pi_\psi(H_m^\circ)$ (satisfying certain twisted endoscopic character identities) which is a finite multi-set consisting of elements of $\Pi^{\Sigma_0}(H_m^\circ)$ (\cite{Art13}). M{\oe}glin gave a construction of $\Pi_\psi(H_m^\circ)$ and showed that it is multiplicity free (\cite{Moe06a, Moe06b, Moe09a, Moe10, Moe11a}). For $\psi\in\Psi^+(H_m)$, we define the local Arthur packet of $H_m$ by $\Pi_\psi(H_m)$ to be the set of all $\pi\in\Pi(H_m)$ such that the restriction of $\pi$ to $H_m^\circ$ has irreducible constituents in $\Pi_\psi(H_m^\circ).$ When there is no ambiguity, we write $\Pi_\psi=\Pi_\psi(H_m).$

Note that $H_m^\circ$ splits over a quadratic extension $E$ of $F$. Let $\Gamma_{E/F}=\mathrm{Gal}(E/F)$. Then $\Gamma_{E/F}\cong\Sigma_0$ and ${}^LH_m^\circ=\widehat{H_m^\circ}(\BC)\rtimes\Gamma_{E/F}= \widehat{H_m}(\BC).$ Let $\xi_m$ be the embedding of the orthogonal group $\widehat{H_m}(\BC)$ into $\GL_m(\BC).$ By composing $\xi_m$ with $\psi\in\Psi^+(H_m)$, we may view $\psi$ as a local Arthur parameter of $\GL_m(F).$ In this way, again by the Local Langlands Correspondence for general linear groups (\cite{HT01, Hen00, Sch13}), we decompose $\psi$ similarly to \eqref{A-param decomp}. 

Let $\psi\in\Psi^+(H_m)$ be decomposed as in \eqref{A-param decomp}. We say that $\psi$ is of \emph{good parity} if $\psi \in \Psi(H_m)$ (i.e. $x_i=0$ for all $i$) and every summand $\rho \otimes S_{a_i} \otimes S_{b_i}$ is self-dual and orthogonal.
We let $\Psi_{gp}(H_m)$ denote the subset of $\Psi(H_m)$ consisting of local Arthur parameters of good parity.

Similarly to the symplectic case, we decompose $\psi\in\Psi^+(H_m)$ as
\begin{equation}\label{eqn psi = nu+np+gp}
    \psi=\psi_{nu,>0}+\psi_{np}+\psi_{gp}+\psi_{np}^\vee+\psi_{nu,>0}^\vee,
\end{equation}
where $\psi_{nu,>0}$ and $\psi_{np}$ are chosen analogously to \eqref{eqn psi_nu,>0} and \eqref{eq decomp of psi_u}, respectively. Similarly to the symplectic case, M{\oe}glin constructed the local Arthur packet $\Pi_{\psi_{np}+\psi_{gp}+\psi_{np}^\vee}$ from $\Pi_{\psi_{gp}}$ (\cite{Moe06a}), i.e., the analogue of Theorem \ref{thm reduction to gp} holds. Furthermore, the analogue of Theorem \ref{thm red from nu to gp symp} also holds as stated below.

\begin{thm}[{\cite[Proposition 5.1]{Moe11b}}]\label{thm red from nu to gp orthog}
Let $\psi\in\Psi^+(H_m)$ with decomposition as in \eqref{eqn psi = nu+np+gp}. Then, for any $\pi_{gp}\in\Pi_{\psi_{gp}},$ the induction $\tau_{\psi_{nu,>0}}\times\tau_{\psi_{np}}\rtimes\pi_{gp}$ is irreducible. Furthermore, \begin{equation}\label{non-unitary A-packet orthog}
    \Pi_\psi=\{\tau_{\psi_{nu,>0}}\times\tau_{\psi_{np}}\rtimes\pi_{gp} , | , \pi_{gp}\in\Pi_{\psi_{gp}}\}.
\end{equation}
Here $\tau_{\psi_{nu,>0}}$ and $\tau_{\psi_{np}}$ are the irreducible representations defined in \eqref{eqn tau_{psi_{nu,>0}}} and \eqref{eqn tau_{psi_{np}}}, respectively.
\end{thm}

\subsection{Raising operators}
The results in this subsection hold for both symplectic and split odd special orthogonal groups defined over $F$; however, as we are only concerned with the symplectic case, the results are stated only for $G_n.$

Atobe gave a reformulation of M{\oe}glin's parameterization of local Arthur packets for split symplectic and odd special orthogonal groups (\cite{Ato20b}). Atobe and independently Liu, Lo, and the author used Atobe's reformulation to define certain operators which systematically determine when two local Arthur packets intersect (\cite{Ato23, HLL22}). The operators of \cite{Ato23} and \cite{HLL22} are different, but logically equivalent; however, the operators of \cite{HLL22} have been used to define certain distinguished members (Theorem \ref{thm psi max min} below) in the set
$$
\Psi(\pi):=\{\psi\in\Psi^+(G_n) \ | \ \pi\in\Pi_\psi\}.
$$
We recall how the operators considered in \cite{HLL22} act on the local Arthur parameters.  

\begin{defn}[{\cite[Definition 12.1]{HLL22}}]\label{def operators on parameters}
Suppose $\psi$ is a local Arthur parameter of $G_n$. We decompose $\psi= \psi_{nu,>0}+\psi_{np}+\psi_{gp}+\psi_{np}^\vee+\psi_{nu,>0}^\vee$ as in Theorem \ref{thm red from nu to gp symp} and write
\[ \psi_{gp}=\bigoplus_{\rho} \bigoplus_{i \in I_{\rho}} \rho \otimes S_{a_i} \otimes S_{b_i}.  \]
Then for $i,j,k \in I_{\rho}$, we define the operators $dual$, $ui_{i,j}$ and $dual_k^{-}$ as follows. 
\begin{enumerate}
    \item $dual(\psi):=\widehat{\psi}$, where $\widehat{\psi}(w,x,y)=\psi(w,y,x).$  We identify the index set $I_{\rho}(\psi_{gp})$ with $I_{\rho}(\widehat{\psi}_{gp})$.
    \item For $r \in I_{\rho}$, let $A_r= \frac{a_r+b_r}{2}-1$ and $B_r= \frac{a_r-b_r}{2}$. Rewrite the decomposition of $\psi_{gp}$ as 
    \[ \psi_{gp}= \bigoplus_{\rho} \bigoplus_{i \in I_\rho} \rho \otimes S_{A_i+B_i+1} \otimes S_{A_i-B_i+1}.\]
    The operator $ui_{i,j}$ is applicable on $\psi$ if the following conditions hold.
    \begin{enumerate}
    \item $A_j \geq A_i+1 \geq B_j >B_i.$
        \item  For any $r \in I_{\rho}$, if $B_i<B_r<B_j$, then $A_r \leq A_i $ or $A_r \geq A_j$.
    \end{enumerate}
    In this case, we define $ui_{i,j}(\psi_{gp})$ by replacing the summands 
    \[ \rho\otimes S_{A_i+B_i+1}\otimes S_{A_i-B_i+1} +\rho\otimes S_{A_j+B_j+1}\otimes S_{A_j-B_j+1} \]
of $\psi_{gp}$ with
    \[\rho\otimes S_{A_j+B_i+1}\otimes S_{A_j-B_i+1} +\rho\otimes S_{A_i+B_j+1}\otimes S_{A_i-B_j+1}. \]
    If $A_i+1-B_j=0$, then we omit the last summand, and say this $ui_{i,j}$ is of type 3'. Finally, we define $ui_{i,j}(\psi)= \psi_{nu,>0}+\psi_{np} + ui_{i,j}(\psi_{gp}) +\psi_{np}^\vee+\psi_{nu,>0}^\vee$.
    \item The operator $dual_k^{-}$ is applicable on $\psi$ if $b_k=a_k+1$. In this case, we define $dual_k^{-}(\psi_{gp})$ by replacing the summand
    \[ \rho \otimes S_{a_k}\otimes S_{a_{k}+1}\]
    of $\psi_{gp}$ with 
    \[ \rho \otimes S_{a_k+1}\otimes S_{a_{k}},\]
    and we define $dual_k^{-}(\psi)= \psi_{nu,>0}+\psi_{np} + dual_k^{-}(\psi_{gp}) +\psi_{np}^\vee+\psi_{nu,>0}^\vee$. 
    \item 
    Let $T$ be any of the operators above or their inverses. If $T$ is not applicable on $\psi$, then we define $T(\psi)=\psi$.
\end{enumerate}
\end{defn}

\begin{rmk}
    We remark that the $B_r$ used above differs from that in M{\oe}glin's parameterization (see \S\ref{subsec Moeglin param}). The difference is that the $B_r$ above may be negative, while in M{\oe}glin's parameterization we take the absolute value.
\end{rmk}

In particular, we wish to study the effects of the raising operators.

\begin{defn}\label{def raising operator}
We say that $T$ is an \emph{raising} operator if it is of the form  
$ ui_{i,j}^{-1}$, $dual \circ ui_{j,i} \circ dual,$ or $dual_k^{-}$.
\end{defn}

Raising operators give a partial order on $\Psi(\pi).$

\begin{defn}\label{def operator ordering}
We define a partial order $\geq_{O}$ on $\Psi(G_n)$ by $\psi_1 \geq_{O} \psi_2$ if $\psi_1=\psi_2$ or there exists a sequence of raising operators $\{T_l\}_{l=1}^m$ such that
\[ \psi_1= T_1 \circ \cdots \circ T_m(\psi_2).\]
\end{defn}

Remarkably, this partial order defines unique maximal and minimal elements, denoted $\psi^{max}(\pi)$ and $ \psi^{min}(\pi)$, in $\Psi(\pi)$.
\begin{thm}[{\cite[Theorem 1.16(2)]{HLL22}}]\label{thm psi max min}
    Let $\pi\in\Pi(G_n)$ be of Arthur type. Then there exists unique maximal and minimal elements denoted by $\psi^{max}(\pi)$ and $ \psi^{min}(\pi)$, respectively, in $\Psi(\pi)$ with respect to the partial order $\geq_O.$
\end{thm}
Moreover, $\psi^{max}(\pi)$ and $ \psi^{min}(\pi)$ are the unique maximal and minimal elements with represent to many other orderings (\cite[Theorem 1.16]{HLL22} and \cite[Theorems 1.9, 1.12]{HLLZ22}).

\subsection{M{\oe}glin's parameterization}\label{subsec Moeglin param}
 Let $G'=G_n,H_m.$ In this section, we review M{\oe}glin's parameterization of the local Arthur packets $\Pi_{\psi}$ for $\psi\in\Psi^+(G')$ (\cite{Moe06a, Moe06b, Moe09a, Moe10, Moe11a}) In view of Theorems \ref{thm red from nu to gp symp} and \ref{thm red from nu to gp orthog}, it is sufficient to give the parameterization for $\psi\in\Psi_{gp}(G').$ We decompose such $\psi$ as 
 \begin{equation}\label{eqn gp decomp}
     \psi:= \bigoplus_{\rho}\left(\bigoplus_{i\in I_\rho} \rho \otimes S_{a_i} \otimes S_{b_i}\right).
 \end{equation}
For $i\in I_\rho,$ we set $d_i=\min(a_i,b_i)$ and $\zeta_i\in\{\pm1\}$ such that $\zeta_i(a_i-b_i)\geq 0.$ If $a_i=b_i,$ we set $\zeta_i=1$ by convention. We also fix a total order $>$ on $I_\rho$ such that
\begin{equation}\label{eqn admissible}
a_i+b_i>a_j+b_j, \, |a_i-b_i|>|a_j-b_j|,\, \mathrm{and}\, \zeta_i=\zeta_j \implies i>j.
\end{equation}
For $i\in I_\rho,$ let $l_i$ to be some integer with $0\leq l_i \leq \frac{d_i}{2}$ and $\eta_i\in\{\pm1\}$ such that the sign condition holds:
\begin{equation}\label{Moeglin sign}
    \prod_{\rho} \prod_{i \in I_\rho} (-1)^{[\frac{d_i}{2}]+l_i} \eta_i^{d_i} = \epsilon_{G'}.
\end{equation}
Here, $\epsilon_{G'}=1$ if $G'=G_n$ is symplectic and is given by the Hasse invariant otherwise. That is, $\epsilon_{H_m}\in\{\pm 1\}$.

Let $\underline{\zeta},\underline{l},\underline{\eta}$ denote the collections of $(\zeta_i)_{\rho,i\in I_\rho}, (l_i)_{\rho,i\in I_\rho},(\eta_i)_{\rho,i\in I_\rho}$ for the various $\rho.$ To such data, M{\oe}glin constructed a representation $\pi_>(\psi,\underline{\zeta},\underline{l},\underline{\eta})$ of $G'$ which either vanishes or belongs to the local Arthur packet $\Pi_{\psi}$ (\cite{Moe06a, Moe06b, Moe09a, Moe10, Moe11a}, see also \cite{Xu17}).

\begin{thm}\label{thm Moeglin construction} Let $\psi$ be a local Arthur parameter.
\begin{enumerate}
    \item The representation $\pi_>(\psi,\underline{\zeta},\underline{l},\underline{\eta})$ is either irreducible or zero.
    \item The local Arthur packet is given by exhausting the above representations, i.e., $\Pi_\psi =\{\pi_>(\psi,\underline{\zeta},\underline{l},\underline{\eta})\}\setminus\{0\}.$
\end{enumerate}
\end{thm}

The following notation and terminology is useful when describing various results, especially in Xu's nonvanishing algorithm (see \S\ref{sec Xu nonvanish}).

Given a summand $\rho\otimes S_{a_i}\otimes S_{b_i}$ of $\psi,$ we define the corresponding Jordan block to be the tuple $(\rho, a_i, b_i)$. We let the set of all Jordan blocks attached to $\psi$ be denoted by $\Jord(\psi).$ We can view the collections $\underline{\zeta},\underline{l},\underline{\eta}$ as functions on $\Jord(\psi)$. For example, we have $\underline{\zeta}(\rho,a_i,b_i)=\zeta_i$. Equivalently, we view  $\Jord(\psi)$ as the set of tuples $(\rho, A_i, B_i, \zeta_i)$ where $A_i=\frac{a_i+b_1}{2}-1$ and $B_i=\frac{|a_i-b_i|}{2}.$ We also set $\Jord_\rho(\psi)=\{(\rho',A,B,\zeta)\in\Jord(\psi) \ | \ \rho'=\rho\}.$

Let $>_\psi$ (which we also denote by $>$ when there is no chance for confusion) denote a total order on $\Jord(\psi).$ We say that the order $>$ is admissible if for any $(\rho,A,B,\zeta),(\rho,A',B',\zeta')\in\Jord(\psi)$ with $A>A',$ $B>B',$ and $\zeta=\zeta'$, we have $(\rho,A,B,\zeta)>(\rho,A',B',\zeta').$ This property is equivalent to \eqref{eqn admissible}.

We say that $\psi$ (or $\Jord(\psi)$) has discrete diagonal restriction if $\psi$ is of good parity and for any $(\rho,A,B,\zeta),(\rho,A',B',\zeta')\in\Jord(\psi)$, the intervals $[B,A]$ and $[B',A']$ do not intersect. Suppose that $\psi$ is of good parity, $\psi_\gg$ is of discrete diagonal restriction, and $>_\psi$ is an admissible order on $\Jord(\psi)$ such that
$$
(\rho,A_i,B_i,\zeta_i)>_\psi(\rho,A_{i-1},B_{i-1},\zeta_{i-1}).
$$
We say that $\psi_\gg$ dominates $\psi$ (or that the Jordan blocks $\Jord(\psi_\gg)$ dominate $\Jord(\psi)$) if there exists an ordering $>_{\psi_\gg}$ on $\Jord(\psi_\gg)$ such that
$$
(\rho,A_{\gg,i},B_{\gg,i},\zeta_{\gg,i})>_{\psi_\gg}(\rho,A_{\gg,i-1},B_{\gg,i-1},\zeta_{\gg,i-1}),
$$
and there exists nonnegative integers $T_i$ such that for $(\rho,A_{\gg,i},B_{\gg,i},\zeta_{\gg,i})\in\Jord(\psi_\gg)$, we have $(\rho,A_{\gg,i},B_{\gg,i},\zeta_{\gg,i})=(\rho,A_i+T_i,B_i+T_i,\zeta_i)$ where $(\rho,A_i+T_i,B_i+T_i,\zeta_i)\in\Jord(\psi)$ and the ordering $>_{\psi_\gg}$ agrees with $>_{\psi}.$


The operators of \cite{HLL22} (and \cite{Ato23}) can be translated from Atobe's parameterization of local Arthur packets to M{\oe}glin's parameterization of local Arthur packets using \cite[Theorem 6.6]{Ato20b}. The details of the definition are not needed for our purpose as the below theorem is sufficient. Given data in M{\oe}glin's parameterization $(\psi,\underline{\zeta},\underline{l},\underline{\eta})$ and operator $T$, we let $T(\psi,\underline{\zeta},\underline{l},\underline{\eta})$ denote the effect of $T$ on M{\oe}glin's parameterization. From \cite[Theorem 1.3]{HLL22} and \cite[Theorem 6.6]{Ato20b}, we obtain the following theorem.

\begin{thm}\label{thm raising operator preserves rep}
    Let $T$ be a raising operator and suppose that $\pi_>(\psi,\underline{\zeta},\underline{l},\underline{\eta})$ and $\pi_>T(\psi,\underline{\zeta},\underline{l},\underline{\eta})$ are representations of $G_n$ which are both non-vanishing. Then we have $\pi_>(\psi,\underline{\zeta},\underline{l},\underline{\eta})=\pi_>T(\psi,\underline{\zeta},\underline{l},\underline{\eta}).$ In other words,
    $$
    \pi_>(\psi,\underline{\zeta},\underline{l},\underline{\eta})\in\Pi_\psi\cap\Pi_{T(\psi)}.
    $$
\end{thm}

The above theorem is true without the restriction that $T$ is a raising operator. In fact, $T$ can be any of the operators in \cite[Theorem 1.3]{HLL22} or \cite[Theorem 1.4]{Ato23}.

One particular operator that we require is the change of order operator, also called the row exchange operator. This was originally defined for M{\oe}glin's construction by Xu (\cite[\S6.1]{Xu21}). We recall the definition below.

\begin{defn}[{\cite[\S6.1]{Xu21}}, Row exchange]\label{defn row exchange} 
Suppose $(\psi_>, \underline{\zeta},\underline{l},\underline{\eta})$ is a part of M{\oe}glin's parameterization of the local Arthur packet $\Pi_\psi.$
Let $>$ be an admissible order on $\Jord(\psi)$ and fix two adjacent Jordan blocks $(\rho,A_{k+1},B_{k+1},\zeta_{k+1})>(\rho,A_{k},B_{k},\zeta_{k})$ in $\Jord(\psi).$ Define an order $\gg$ on $\Jord(\psi)$ by the order coming from $>$ by switching $(\rho,A_k,B_k,\zeta_k)$ and $ (\rho,A_{k+1},B_{k+1},\zeta_{k+1}).$
If $\gg$ is not an admissible order on $\Jord(\psi)$, then we define $R_k(\psi, \underline{\zeta},\underline{l},\underline{\eta})=(\psi, \underline{\zeta},\underline{l},\underline{\eta})$. Otherwise, we define 
\[R_{k}(\psi, \underline{\zeta},\underline{l},\underline{\eta})=(\psi, \underline{\zeta'},\underline{l'},\underline{\eta}'),\]
where $( l_i',\eta_i')=(l_i,\eta_i)$ for $i \neq k,k+1$, and $(l_k',\eta_k')$ and $(l_{k+1}', \eta_{k+1}')$ are given as follows:
\begin{enumerate}
    \item [Case 1.] Assume that $\zeta_{k+1}=\zeta_{k}$ and $[A_k,B_k]_{\rho} \subseteq [A_{k+1},B_{k+1}]_{\rho}.$
\begin{enumerate}
    \item [(a)] If $\eta_{k+1}\neq(-1)^{A_{k}-B_k}\eta_k$, then
    \begin{align*}
        &(l_{k+1}', \eta_{k+1}', l_k', \eta_{k}')\\
        = &(l_{k+1}-1-(A_k-B_k-2l_k),(-1)^{B_{k+1}-A_{k+1}}\eta_{k}',l_k,(-1)^{A_{k+1}-B_{k+1}}\eta_k).
    \end{align*}
    \item [(b)] If $\eta_{k+1}=(-1)^{A_{k}-B_k}\eta_k$, and $l_{k+1}-l_k < \frac{A_{k+1}-B_{k+1}}{2}-(A_k-B_k)+l_k $, then
    \begin{align*} &(l_{k+1}', \eta_{k+1}', l_k', \eta_{k}')\\
    = &(l_{k+1}+1+(A_k-B_k-2l_k),(-1)^{B_{k+1}-A_{k+1}+1}\eta_{k}',l_k,(-1)^{A_{k+1}-B_{k+1}}\eta_k).  \end{align*}
    \item [(c)] If $\eta_{k+1}=(-1)^{A_{k}-B_k}\eta_k$, and $l_{k+1}-l_k \geq \frac{A_{k+1}-B_{k+1}}{2}-(A_k-B_k)+l_k $, then
    \begin{align*}
        &(l_{k+1}', \eta_{k+1}', l_k', \eta_{k}') \\
        = &(2l_{k}-l_{k+1}+A_{k+1}-A_k-B_{k+1}+B_k,(-1)^{B_{k+1}-A_{k+1}}\eta_{k}',l_k,(-1)^{A_{k+1}-B_{k+1}}\eta_k). 
    \end{align*}
\end{enumerate}
We also denote this transformation by $S^-.$

    \item [Case 2.] Assume that $\zeta_{k+1}=\zeta_{k}$ and $[A_k,B_k]_{\rho} \supseteq [A_{k+1},B_{k+1}]_{\rho}.$ In this case, we simply reverse the construction the the previous case. We also denote this transformation by $S^+.$
    
    \item [Case 3.] Assume that $\zeta_{k+1}\neq\zeta_{k}.$ Then
     \[ (l_{k+1}', \eta_{k+1}', l_k', \eta_{k}')= (l_{k+1},(-1)^{B_{k}-A_{k}-1}\eta_{k+1},l_k,(-1)^{B_{k+1}-A_{k+1}-1}\eta_k).  \]
\end{enumerate}
\end{defn}

Xu proved that changing the order in this way preserves the representation.

\begin{thm}[{\cite[Theorems 6.2, 6.3]{Xu21}}]\label{thm Moeglin row exchange equiv}
    For any quasi-split symplectic or special orthogonal group, we have $\pi_>(\psi,\underline{\zeta},\underline{l},\underline{\eta})=\pi_\gg(R_k(\psi, \underline{\zeta},\underline{l},\underline{\eta})).$
\end{thm}

\subsection{The {A}dams conjecture}\label{sec Adams conjecture}
As in \cite[\S3.2]{GI14}, we fix a pair of characters $\chi_W,\chi_V$ associated to $W_n$ and $V_m$ respectively (technically this does not depend on their dimensions, but on the Witt tower that they lie in). More specifically, $\chi_W$ is the trivial character of $F^\times$ and $\chi_V$ is quadratic character associated to $F(\sqrt{\mathrm{disc}V_m})/F.$ Note that for any $V\in\mathcal{V}^+\cup\mathcal{V}^-$ (see \eqref{eqn tower+}, \eqref{eqn tower-}), we have $\mathrm{disc}V=d\not\in (F^\times)^2.$ Given a fixed $\pi\in\Pi(G_n)$ we let $\theta_{-\alpha}^{\pm}(\pi)=\theta_{W_n,V_m^\pm,\psi_F}$ be its little theta lift with respect to $V_m^\pm\in\mathcal{V}^\pm$ where $\dim V_m^\pm=m$ and $\alpha=m-n-1$ is a positive odd integer. When it is clear in context which tower is the target tower, we will often drop the $\pm.$

It is known that the theta correspondence does not preserve $L$-packets. As a remedy, Adams proposed that instead, it should preserve local Arthur packets (\cite{Ada89}). We recall Adams' conjecture below. We remark that Adams' conjecture is more broadly stated for other dual pairs; however, we only concern ourselves with the dual pair $(G_n, H_m)$.

\begin{conj}[The Adams conjecture, {\cite[Conjecture A]{Ada89}}]\label{conj Adams}
    Let $\pi\in\Pi(G_n)$ and suppose that $\pi\in\Pi_\psi$ for some $\psi\in\Psi(G_n).$ If $\theta_{-\alpha}(\pi)\neq 0,$ then $\theta_{-\alpha}(\pi)\in\Pi_{\psi_\alpha}$ where
    \begin{equation}\label{eqn psi_alpha}
    \psi_\alpha=(\chi_W\chi_V^{-1}\otimes\psi)\oplus \chi_W\otimes S_1\otimes S_\alpha.
    \end{equation}
\end{conj}

M{\oe}glin verified that for large $\alpha$, the Adams conjecture is true.

\begin{thm}[{\cite[Theorem 6.1]{Moe11c}}]\label{thm Moeglin Adams}
    Let $\alpha\gg 0$. Suppose that $\pi\in\Pi_\psi$ for some $\psi\in\Psi_{gp}(G_n)$ and that $\pi=\pi_>(\psi,\underline{\zeta},\underline{l},\underline{\eta}).$ Then $\theta_{-\alpha}(\pi)\in\Pi_{\psi_\alpha}$. Moreover, $\theta_{-\alpha}(\pi)=\pi_>(\psi_\alpha,\underline{\zeta}',\underline{l}',\underline{\eta}')$ where $\underline{l}'$ and $\underline{\eta}'$ agree with $\underline{l}$ and $\underline{\eta}$ on the Jordan blocks of $\psi_\alpha$ that are coming from $\psi.$
\end{thm}

\begin{rmk}
    On the Jordan block corresponding to the summand $\chi_W\otimes S_1\otimes S_\alpha$ of $\psi_\alpha$, we must have $l'(\chi_W,1,\alpha)=0$; however, $\eta'(\chi_W,1,\alpha)$ is determined by the target tower $V\in\mathcal{V}^\pm$ via \eqref{Moeglin sign}.
\end{rmk}

M{\oe}glin also showed that the Adams conjecture is false for many examples (\cite{Moe11c}). Nevertheless, we are interested in computing its failure. For $\pi\in\Pi_\psi$, we set
\begin{equation}
   \mathcal{A}(\pi,\psi):=\{\alpha\geq 0, \alpha\equiv 1 (\mathrm{mod} \ 2) \ | \ \theta_{-\alpha}(\pi)\in\Pi_{\psi_{\alpha}}\}. 
\end{equation}
Let $a(\pi,\psi)=\min\mathcal{A}(\pi,\psi)$. M{\oe}glin posed the following question.

\begin{ques}[{\cite[\S6.3]{Moe11c}}]\label{ques Moeglin}
   Is there a way to compute $a(\pi,\psi)$ explicitly?
\end{ques}

Baki{\' c} and Hanzer translated this problem into determining whether certain non-vanishing conditions hold as we explain below (\cite{BH22}).

\begin{recipe}\label{rec pi_alpha}
    Suppose that $\pi\in\Pi_{\psi}.$ Let $\alpha\gg 0$ such that $\theta_{-\alpha}(\pi)=\pi_\alpha$ where $\pi_\alpha=\pi_>(\psi_\alpha,\underline{\zeta}',\underline{l}',\underline{\eta}')\in\Pi_{\psi_{\alpha}}$ is given as in Theorem \ref{thm Moeglin Adams}. For $i\geq1,$ we define local Arthur parameters $\psi_{\alpha-2i}$ and representations $\pi_{\alpha-2i}\in\Pi_{\psi_{\alpha-2i}}$ inductively as follows. First, we set
$$
\psi_{\alpha-2i}=(\chi_W\chi_V^{-1}\otimes\psi)\oplus \chi_W\otimes S_1\otimes S_{\alpha-2i},
$$
i.e., $\psi_{\alpha-2i}$ differs from $\psi_{\alpha-2(i+1)}$ by shifting the added summand $\chi_W\otimes S_1\otimes S_{\alpha-2i}$ to $\chi_W\otimes S_1\otimes S_{\alpha-2(i+1)}.$ The parameterization of $\pi_{\alpha-2(i+1)}$ does not change from that of $\pi_{\alpha-2i}$ except in the case that $(\chi_W, A_i, B_i, \zeta_i)\in\Jord(\psi)$ where $B_i=\frac{\beta-1}{2}.$ In this case, we change the order $>$ of $\psi_{\alpha-2(i+1)}$ so that the added block is less than the Jordan block corresponding to $(\chi_W, A_i, B_i, \zeta_i).$ If there are multiple such blocks, then we change the order so that the added block is less than all of them. In this case, $\underline{l}$ and $\underline{\eta}$ change according to Definition \ref{defn row exchange}.
\end{recipe}

Baki{\' c} and Hanzer used this recipe to relate the image of the theta correspondence with M{\oe}glin's parameterization of local Arthur packets.

\begin{thm}[{\cite[Theorem A]{BH22}}]\label{thm A BH22}
    Let $\pi\in\Pi_\psi.$ If $\pi_\alpha=\theta_{-\alpha}(\pi)$ and $\pi_{\alpha-2}\neq0,$ then $\theta_{-(\alpha-2)}(\pi)=\pi_{\alpha-2}.$ In particular, $\theta_{-(\alpha-2)}(\pi)\in\Pi_{\psi_{\alpha+2}}.$
\end{thm}

In view of this theorem, we set
\begin{equation}
    d(\pi,\psi):=\min\{\alpha_0\geq 0\ | \ \pi_\alpha\neq 0 \ \mathrm{for \ all} \ \alpha\geq\alpha_0 \}.
\end{equation}

Baki{\' c} and Hanzer used $d(\pi,\psi)$ to study Adams' conjecture (Conjecture \ref{conj Adams}) and hence M{\oe}glin's Question \ref{ques Moeglin}.
\begin{thm}[{\cite[Corollary D]{BH22}}]\label{thm a=d}
    The Adams conjecture fails for $\alpha<d(\pi,\psi)$. Moreover, $a(\pi,\psi)=d(\pi,\psi).$
\end{thm}
Therefore Question \ref{ques Moeglin} becomes the following.
\begin{ques}\label{ques Moeglin d(pi,psi)}
    Is there a way to compute $d(\pi,\psi)$ explicitly?
\end{ques}

When our target tower is $\mathcal{V}^\pm$ we denote $d(\pi,\psi)$ by $d^\pm(\pi,\psi)$ respectively. In the going up tower $\mathcal{V}^+,$ Baki{\'c} and Hanzer answered Question \ref{ques Moeglin d(pi,psi)} in terms of the first occurrence.

\begin{thm}[{\cite[Theorem 2]{BH22}}]\label{thm BH up down}
     Let $\pi\in\Pi_\psi.$ Then the  Adams conjecture is true for any nonzero lift, i.e.,
     $$
     d^+(\pi,\psi)=m^{+,\alpha}(\pi).
     $$
     Moreover, $d^-(\pi,\psi)<d^+(\pi,\psi).$
\end{thm}

\begin{rmk}\label{rmk Adams up=down}
    Suppose that $m^+(\pi)= n+2=m^-(\pi).$ As explained in \cite[p. 15]{BH22}, Theorem \ref{thm BH up down} then implies that for any $\psi\in\psi(\pi)$ we have  $d(\pi,\psi)=1$ for either tower and so the Adams conjecture (Conjecture \ref{conj Adams}) holds in full generality, i.e., for any $\alpha\geq 1.$ Note that this implies that the inequalities in Theorem \ref{thm main thm} and \ref{thm main max thm} below are all equalities in this case.
\end{rmk}

Theorem \ref{thm BH up down} implies that $d^+(\pi,\psi)=d^+(\pi,\psi')$ for any $\psi,\psi'\in\Psi(\pi)$. This is not the expectation for the going down tower. Baki{\' c} and Hanzer suspected that $d(\pi,\psi)$ gets smaller as $\psi$ gets more tempered (\cite[p. 5]{BH22}). The main theorem of this paper is to verify that this is indeed the case.

\begin{thm}\label{thm main thm}
    Let $\pi$ be a representation of $G_n$ of Arthur type. Suppose that $T$ is a raising operator and $\pi\in\Pi_\psi\cap\Pi_{T(\psi)}.$ Then
    \begin{equation}
        d^-(\pi,T(\psi))\leq d^-(\pi,\psi).
    \end{equation}
\end{thm}

Theorem \ref{thm main thm} follows from Theorem \ref{thm alpha neq 0 implies Talpha neq 0}. From Theorem \ref{thm psi max min}, we obtain the following immediately.

\begin{thm}\label{thm main max thm}
    Let $\pi$ be a representation of $G_n$ of Arthur type. Then for any $\psi\in\Psi(\pi)$, we have
    \begin{equation}
        d^-(\pi,\psi^{max}(\pi))\leq d^-(\pi,\psi)\leq d^-(\pi,\psi^{min}(\pi)).
    \end{equation}
\end{thm}

Note that Theorems \ref{thm a=d} and \ref{thm BH up down} show that $a(\pi,\psi)=m^+(\pi)$ on the going up tower. On the going down tower, we only know by Theorem \ref{thm a=d} that $a(\pi,\psi)=d^-(\pi,\psi)$. Furthermore, M{\oe}glin showed there exists a representation $\pi$ of Arthur type which has a theta lift that is not of Arthur type (\cite[Remark 8.1]{Moe11c}). Consequently, in contrast to the going up tower, we do not except that $d^-(\pi,\psi^{max}(\pi))=m^{-,\alpha}(\pi).$ Instead, we expect that $d^-(\pi,\psi^{max}(\pi))$ detects when the theta lift is stably of Arthur type. Let $m_A^\pm(\pi)$ be the minimal $2(1+r)$ such that $\theta_{W_n,V_r^\pm,\psi_F}(\pi)\neq 0$ and $\theta_{W_n,V_{r'}^\pm,\psi_F}(\pi)$ is of Arthur type for any $r'\geq r.$ We remark that $m_A^\pm(\pi)$ is well defined by the following theorem of Baki{\'c} and Hanzer.

\begin{thm}[{\cite[Theorem C]{BH22}}]\label{thm mA well-def} 
Suppose that $\pi\in\Pi_\psi$ for some $\psi\in\Psi(G_n).$ 
    If $\theta_{-\alpha}(\pi)=\pi_{\alpha}\in\Pi_{\psi_\alpha},$ then $\pi_{\alpha+2}=\theta_{-(\alpha+2)}(\pi)\in\Pi_{\psi_{\alpha+2}}.$
\end{thm}

Recall that $\alpha=m-n-1.$ Set $m^{\pm,\alpha}_A(\pi)=m_A^{\pm}(\pi)-n-1.$ We conjecture the following.

\begin{conj}\label{conj stable Arthur}
    Let $\pi\in\Pi(G).$ If $d^\pm(\pi,\psi^{max}(\pi))>1,$ then 
    $$
    d^\pm(\pi,\psi^{max}(\pi))=m_A^{\pm,\alpha}(\pi).
    $$
\end{conj}
By Theorems \ref{thm BH up down} and \ref{thm mA well-def}, it follows that in the going up tower, $m_A^{+,\alpha}(\pi)=m^{+,\alpha}(\pi)$ and so the conjecture is true there. However, in the going down tower, it is possible there is a `gap' in the theta lift being of Arthur type, i.e., $m^-(\pi)<m_A^-(\pi).$ We explicate this in Example \ref{exmp stable Arthur}. Also, note that by definition
$$
d^-(\pi,\psi^{max}(\pi))\geq m_A^{-,\alpha}(\pi).
$$

In Example \ref{exmp supercuspidal}, we show the conjecture is false if we drop the condition that $d^\pm(\pi,\psi^{max}(\pi))>1.$ This is not too concerning. Indeed, if $d^\pm(\pi,\psi^{max}(\pi))=1,$ then the Adams conjecture for $\psi^{max}(\pi)$ holds in all the ranks it predicts. If Conjecture \ref{conj stable Arthur} failed, then $m_A^{\pm,\alpha}(\pi)$ must be an odd negative integer which is not controlled by the Adams conjecture. It is thus an interesting question to extend the Adams conjecture for $\alpha<0.$

\subsection{Reduction to good parity}

In this subsection, we reduce Theorem \ref{thm main thm} to the good parity case. The key step in this reduction is given by the following lemma which is a consequence of Kudla's filtration (\cite[Theorem 2.8]{Kud86}).

\begin{lemma}[{\cite[Remark 3.10]{BH22}}]\label{lem Kudla remark}
     Let $\pi\in\Pi(G_n)$, $\pi_0\in\Pi(G_{n-2k}),$ $\alpha=n-m+1,$ and $\sigma$ be a multisegment representation such that the corresponding multisegment does not contain $\frac{\alpha-1}{2}.$ Let $P$ be the maximal parabolic subgroup associated to $\chi_V\sigma\rtimes\Theta_{-\alpha}(\pi_0)$. Furthermore, we assume that
     \begin{enumerate}
         \item $r_{\overline{P}}(\theta_{-\alpha}(\pi))$ has a unique irreducible subquotient on which $\GL_k(F)$ acts by $\chi_W\sigma$; or
         \item $r_{\overline{P}}(\pi)$ has a unique irreducible subquotient on which $\GL_k(F)$ acts by $\chi_V\sigma$.
     \end{enumerate}
     
     If $\chi_V\sigma\rtimes\pi_0\twoheadrightarrow \pi,$ then $\chi_W\sigma\rtimes\theta_{-\alpha}(\pi_0)\rightarrow \theta_{-\alpha}(\pi).$
\end{lemma}

As mentioned in \cite[\S1.6]{BH22}, the above lemma reduces the study of Adams' conjecture from the not good parity case to that of good parity. However, the argument also reduces the general local Arthur parameter case, i.e., $\psi\in\Psi^+(G_n)$, to those of good parity. For completeness, we give the argument here.

\begin{lemma}\label{lem Red of Adams to gp}
    Suppose that Theorem \ref{thm main thm} holds for local Arthur parameters of good parity. Then Theorem \ref{thm main thm} holds for general local Arthur parameters.
\end{lemma}

\begin{proof}
    Let $\pi$ be a representation of $G_n$ such that $\pi\in\Pi_\psi$ for some $\psi\in\Psi^+(G_n).$ Furthermore assume that $\theta_{-\alpha}(\pi)\neq 0.$ Decompose $\psi$ as in \eqref{eqn psi decomp general}. By Theorem \ref{thm red from nu to gp symp}, we have $\pi=\tau_{\psi_{nu,>0}}\times\tau_{\psi_{np}}\rtimes\pi_{gp}$ for some $\pi_{gp}\in\Pi_{\psi_{gp}}.$ Let $\alpha\gg 0$ such that  $\theta_{-\alpha}(\pi)=\pi_\alpha\in\Pi_{\psi_\alpha}$ and $\theta_{-\alpha}(\pi_{gp})=(\pi_{gp})_\alpha\in\Pi_{\psi_\alpha}$ by Theorem \ref{thm Moeglin Adams} and also the multisegments corresponding to $\tau_{\psi_{nu,>0}}\times\tau_{\psi_{np}}$ do not contain $\frac{\alpha-1}{2}$. By Theorem \ref{thm red from nu to gp orthog},  we have $\pi_\alpha=\chi_V\inv\tau_{\psi_{nu,>0}}\times\chi_V\inv\tau_{\psi_{np}}\rtimes\pi'_{gp}$ for some $\pi'_{gp}\in\Pi_{(\psi_{\alpha})_{gp}}.$ By Frobenius reciprocity (Lemma \ref{lemma Frobenius rec}), we have that Condition (1) of Lemma \ref{lem Kudla remark} holds. Thus we have a nonzero map
    $$
    \chi_V\inv\tau_{\psi_{nu,>0}}\times\chi_V\inv\tau_{\psi_{np}}\rtimes\theta_{-\alpha}(\pi_{gp})\rightarrow \theta_{-\alpha}(\pi).
    $$
    From Theorem \ref{thm red from nu to gp orthog}, $\chi_V\inv\tau_{\psi_{nu,>0}}\times\chi_V\inv\tau_{\psi_{np}}\rtimes\theta_{-\alpha}(\pi_{gp})$ is irreducible and hence 
    \begin{align*}
        \chi_V\inv\tau_{\psi_{nu,>0}}\times\chi_V\inv\tau_{\psi_{np}}\rtimes\theta_{-\alpha}(\pi_{gp})= \theta_{-\alpha}(\pi).
    \end{align*}
    Hence the parameterization of $\pi_\alpha$ in $\Pi_{\psi_\alpha}$ is exactly the parameterization of $(\pi_{gp})_\alpha$ in $\Pi_{(\psi_\alpha)_{gp}}$. The lemma then follows directly from the definition of raising operators (Definition \ref{def raising operator}) and the assumption that Theorem \ref{thm main thm} holds in the good parity case.
\end{proof}

\section{Xu's nonvanishing algorithm}\label{sec Xu nonvanish}

In this section, we explain Xu's nonvanishing algorithm. First, we recall the relevant terms and operators involved in the algorithm.

\begin{defn}[{\cite[\S2]{Xu21}}]
    Let $\psi$ be a local Arthur parameter of good parity and fix $\rho$ to be a self-dual irreducible unitary supercuspidal representation of $\GL_d(F)$ for some $d\in\mathbb{Z}_{\geq 1}.$ 
    \begin{enumerate}
        \item Let $(\rho,A,B,\zeta)\in\Jord_\rho(\psi)$ and $r\in\mathbb{Z}_{\geq 1}.$ We say $(\rho,A,B,\zeta)$ (or $[A,B]$ for brevity) is far away (of level $r$) from a subset $J\subseteq\Jord_\rho(\psi)$ if
        $$
        B>2^{r|J|}\left(\sum_{(\rho,A',B',\zeta')\in J} \left( A'+|J|\right) +\sum_{(\rho,A',B',\zeta')\in\Jord_\rho(\psi)}\left( A'-B'+1 \right)\right).
        $$
        In this case we write
        $$
        (\rho,A,B,\zeta)\gg_r J.
        $$
        \item Let $J\subseteq\Jord_\rho(\psi)$ and $J^c=\Jord_\rho(\psi)\setminus J.$ We say $J$ is separated from $J^c$ if the following conditions hold.
        \begin{enumerate}
            \item For any $(\rho,A,B,\zeta)\in J$ and $(\rho,A',B',\zeta')\in J^c,$ we have either
            $$
            B>A' \ \mathrm{or} \ B'>A.
            $$
            \item For any admissible order on $J,$ there exists a dominating set of Jordan blocks $J_\gg$ of $J$ with discrete diagonal restriction, such that for any $(\rho,A,B,\zeta)\in J$ and $(\rho,A',B',\zeta')\in J^c,$ if $B'>A$, then $B'>A_\gg.$
            \item There exists admissible order on $J^c$ such that there exists a dominating set of Jordan blocks $J_\gg^c$ of $J^c$ with discrete diagonal restriction, such that for any $(\rho,A,B,\zeta)\in J$ and $(\rho,A',B',\zeta')\in J^c,$ if $B>A'$, then $B>A'_\gg.$
        \end{enumerate}
    \end{enumerate}
\end{defn}

The following definition will be the end result of apply Xu's nonvanishing algorithm.

\begin{defn}[{\cite[\S5]{Xu21}}]
    Let $\psi$ be a local Arthur parameter and suppose that we can index $\Jord_\rho(\psi)$ for each $\rho$ such that $A_i\geq A_{i-1}$ and $B_i\geq B_{i-1}.$ We say that $\psi$ is in the generalized basic case if, for each $\rho$, $\Jord_\rho(\psi)$ can be divided into subsets
    \begin{equation}\label{eqn pair in generalized basic}
        \{(\rho,A_i,B_i,\zeta_i),(\rho,A_{i-1},B_{i-1},\zeta_{i-1})\} \ \mathrm{such} \ \mathrm{that} \ \zeta_i=\zeta_{i-1},
    \end{equation}
    or
    \begin{equation}
        \{(\rho,A_i,B_i,\zeta_i)\}, 
    \end{equation}
    such that each subset is separated from the others.

    In general, if a subset of $\Jord_\rho(\psi)$ satisfies the conditions to be in the generalized basic case, then we say that this subset is in good shape.
\end{defn}

Xu showed that if $\psi$ is in the generalized basic case, then the nonvanishing of $\pi_>(\psi,\underline{\zeta},\underline{l},\underline{\eta})$ is determined purely combinatorially.

\begin{prop}[{\cite[Proposition 5.3]{Xu21}}]\label{prop generalized basic case}
    Let $\psi$ be in the generalized basic case. Then $\pi_>(\psi,\underline{\zeta},\underline{l},\underline{\eta})\neq 0$ if and only if for any pair $(\rho,A_i,B_i,\zeta_i),(\rho,A_{i-1},B_{i-1},\zeta_{i-1})$ in \eqref{eqn pair in generalized basic}, we have
    \begin{equation}\label{eqn nonvanishing separated}
\left\{
    \begin{array}{ll}
        \mathrm{if} \ \eta_i=(-1)^{A_{i-1}-B_{i-1}}\eta_{i-1}, & \mathrm{then} \ A_i-l_i\geq A_{i-1}-l_{i-1} \\
        & \mathrm{and} \ B_i+l_i\geq B_{i-1}+l_{i-1},\\
        \mathrm{if} \ \eta_i\neq(-1)^{A_{i-1}-B_{i-1}}\eta_{i-1}, & \mathrm{then} \ B_i+l_i\geq A_{i-1}-l_{i-1}. 
    \end{array}
\right.
\end{equation}
\end{prop}

Next we describe the `pull' operator. The idea is to `pull' (or shift) a pair of Jordan blocks away from the others so that it becomes separated. There are two settings in which we define the pull operator depending on when the Jordan blocks are of unequal length or not.

Let $\psi$ be a local Arthur parameter of good parity, $>_\psi$ be an admissible order, and fix a self-dual irreducible unitary supercuspidal representation $\rho$ of $\GL_d(F).$ Index $\Jord_\rho(\psi)$ such that $(\rho,A_i,B_i,\zeta_i)>(\rho,A_{i-1},B_{i-1},\zeta_{i-1}).$ Suppose there exists $n$ such that for $i>n$, we have
$$
(\rho,A_i,B_i,\zeta_i)\gg_1 \bigcup_{j=1}^n\{(\rho,A_j,B_j,\zeta_j)\}.
$$
Moreover, we assume that 
$$[A_{n-1},B_{n-1}]\subsetneq [A_n,B_n] \ \mathrm{and} \ \zeta_{n-1}=\zeta_n.$$
This is the condition of unequal length. Let $>'$ denote the order obtained from $>$ by switching $n$ and $n-1.$ Note that this is still an admissible order. Let $S^+(\underline{\zeta},\underline{l},\underline{\eta})$ denote the corresponding effect of this switch (see Definition \ref{defn row exchange}). Define $\psi_-$ by
$$
\Jord(\psi)=\Jord(\psi)\setminus\{(\rho,A_n,B_n,\zeta_n),(\rho,A_{n-1},B_{n-1},\zeta_{n-1})\}
$$
and define $(\underline{\zeta}_-,\underline{l}_-, \underline{\eta}_-)$ to be the restriction of $(\underline{\zeta},\underline{l}, \underline{\eta})$ to $\Jord(\psi_-).$ Given arbitrary Jordan blocks $(\rho, A,B,\zeta),(\rho, A',B',\zeta')$ along with corresponding functions $(l,\eta)$ and $(l',\eta')$, we define a (possibly vanishing) representation
\[\pi_>(\psi_-,\underline{\zeta}_-,\underline{l}_-,\underline{\eta}_-;(\rho, A,B,\zeta,l,\eta),(\rho, A',B',\zeta',l',\eta'))\]
to be the representation defined by M{\oe}glin in Theorem \ref{thm Moeglin construction} which is obtained by replacing the Jordan blocks $(\rho,A_n,B_n,\zeta_n),(\rho,A_{n-1},B_{n-1},\zeta_{n-1})$ in $\Jord(\psi)$ with $(\rho, A,B,\zeta),(\rho, A',B',\zeta')$, respectively. Note that this ordering may no longer be admissible, but in the cases we consider below it always will be.

Xu showed that in the above setting the nonvanishing of the representation depends on the nonvanishing of $3$ related representations which are defined to be `pull' of the Jordan blocks $(\rho,A_n,B_n,\zeta_n),(\rho,A_{n-1},B_{n-1},\zeta_{n-1}).$

\begin{prop}[{\cite[Proposition 7.1]{Xu21}}]\label{prop pull unequal case}
    Let $\psi$ and $\psi_-$ be as above. Then $\pi_>(\psi,\underline{\zeta},\underline{l},\underline{\eta})\neq 0$ if the following hold:
    \begin{enumerate}
        \item We have \begin{align*}
            \pi_>&(\psi_-,\underline{\zeta}_-,\underline{l}_-,\underline{\eta}_-;(\rho, A_n+T_n,B_n+T_n,\zeta_n,l_n,\eta_n), \\&(\rho, A_{n-1}+T_{n-1},B_{n-1}+T_{n-1},\zeta_{n-1},l_{n-1},\eta_{n-1}))\neq 0
        \end{align*}
        for some nonnegative integers $T_n, T_{n-1}$ such that 
        $$
        [A_{n-1}+T_{n-1},B_{n-1}+T_{n-1}]\subsetneq[A_n+T_n,B_n+T_n]
        $$
        and for any $i>n$, we have $(\rho,A_i,B_i,\zeta_i)\gg_1(\rho, A_n+T_n,B_n+T_n,\zeta).$\item We have \begin{align*}
            \pi_>&(\psi_-,\underline{\zeta}_-,\underline{l}_-,\underline{\eta}_-;(\rho, A_n+T_n,B_n+T_n,\zeta_n,l_n,\eta_n), \\&(\rho, A_{n-1},B_{n-1},\zeta_{n-1},l_{n-1},\eta_{n-1}))\neq 0,
        \end{align*}
        for some nonnegative integer $T_n$ such that 
        for any $i>n$, we have that $B_i> A_n+T_n.$
        \item We have \begin{align*}
            \pi_{>'}&(\psi_-,\underline{\zeta}'_-,\underline{l}'_-,\underline{\eta}'_-;(\rho, A_n,B_n,\zeta'_n,l'_n,\eta'_n), \\&(\rho, A_{n-1}+T_{n-1},B_{n-1}+T_{n-1},\zeta'_{n-1},l'_{n-1},\eta'_{n-1}))\neq 0,
        \end{align*}
        for some nonnegative integer $T_{n-1}$ such that 
        for any $i>n$, we have $B_i>A_{i-1}+T_{i-1}$ and $(\underline{\zeta}',\underline{l}',\underline{\eta}')=S^+(\underline{\zeta},\underline{l},\underline{\eta})$ (see Definition \ref{defn row exchange}).
    \end{enumerate}
    Conversely, if $\pi_>(\psi,\underline{\zeta},\underline{l},\underline{\eta})\neq 0$, then (1), (2), and (3) hold with ``for some'' replaced with ``for all.''
\end{prop}

Each of the resulting parameterizations in (1), (2), and (3) above are called the pull of $(\rho,A_n,B_n,\zeta_n),(\rho,A_{n-1},B_{n-1},\zeta_{n-1}).$

Next we define the pull operator in the case of equal length. Let $\psi$ be a local Arthur parameter of good parity, $>_\psi$ be an admissible order, and fix a self-dual irreducible unitary supercuspidal representation $\rho$ of $\GL_d(F).$ Index $\Jord_\rho(\psi)$ such that $(\rho,A_i,B_i,\zeta_i)>(\rho,A_{i-1},B_{i-1},\zeta_{i-1}).$ Suppose there exists $n$ such that for $i>n$, we have
$$
(\rho,A_i,B_i,\zeta_i)\gg_1 \bigcup_{j=1}^n\{(\rho,A_j,B_j,\zeta_j)\}.
$$
Moreover, we assume that 
$$[A_{n-1},B_{n-1}]= [A_n,B_n] \ \mathrm{and} \ \zeta_{n-1}=\zeta_n.$$
This is the condition of equal length. Define $\psi_-$ by
$$
\Jord(\psi_-)=\Jord(\psi)\setminus\{(\rho,A_n,B_n,\zeta_n),(\rho,A_{n-1},B_{n-1},\zeta_{n-1})\}
$$
and define $(\underline{\zeta}_-,\underline{l}_-, \underline{\eta}_-)$ to be the restriction of $(\underline{\zeta},\underline{l}, \underline{\eta})$ to $\Jord(\psi_-).$

Again, Xu showed that in the above setting the nonvanishing of the representation depends on the nonvanishing of $2$ related representations which are defined to be `pull' of the Jordan blocks $(\rho,A_n,B_n,\zeta_n),(\rho,A_{n-1},B_{n-1},\zeta_{n-1}).$

\begin{prop}[{\cite[Proposition 7.3]{Xu21}}]\label{prop pull equal case}
    Let $\psi$ and $\psi_-$ be as above. Then $\pi_>(\psi,\underline{\zeta},\underline{l},\underline{\eta})\neq 0$ if the following hold:
    \begin{enumerate}
        \item We have \begin{align*}
            \pi_>&(\psi_-,\underline{\zeta}_-,\underline{l}_-,\underline{\eta}_-;(\rho, A_n+T_n,B_n+T_n,\zeta_n,l_n,\eta_n), \\&(\rho, A_{n-1}+T_{n-1},B_{n-1}+T_{n-1},\zeta_{n-1},l_{n-1},\eta_{n-1}))\neq 0
        \end{align*}
        for some nonnegative integer $T_n=T_{n-1}$ such that 
        for any $i>n$, we have $(\rho,A_i,B_i,\zeta_i)\gg_1(\rho, A_n+T_n,B_n+T_n,\zeta).$\item We have \begin{align*}
            \pi_>&(\psi_-,\underline{\zeta}_-,\underline{l}_-,\underline{\eta}_-;(\rho, A_n+T_n,B_n+T_n,\zeta_n,l_n,\eta_n), \\&(\rho, A_{n-1},B_{n-1},\zeta_{n-1},l_{n-1},\eta_{n-1}))\neq 0
        \end{align*}
        for some nonnegative integer $T_n$ such that 
        for any $i>n$, we have that $B_i> A_n+T_n.$
    \end{enumerate}
    Conversely, if $\pi_>(\psi,\underline{\zeta},\underline{l},\underline{\eta})\neq 0$, then (1) and (2) hold with ``for some'' replaced with ``for all.''
\end{prop}

Each of the resulting parameterizations in (1) and (2) above are called the pull of $(\rho,A_n,B_n,\zeta_n),(\rho,A_{n-1},B_{n-1},\zeta_{n-1}).$

The next operator we need is called `expand.' Let $\psi$ be a local Arthur parameter of good parity, $>_\psi$ be an admissible order, and fix a self-dual irreducible unitary supercuspidal representation $\rho$ of $\GL_d(F).$ Index $\Jord_\rho(\psi)$ such that $(\rho,A_i,B_i,\zeta_i)>(\rho,A_{i-1},B_{i-1},\zeta_{i-1}).$ Suppose there exists $n$ such that for $i>n$, we have
$$
(\rho,A_i,B_i,\zeta_i)\gg_2 \bigcup_{j=1}^n\{(\rho,A_j,B_j,\zeta_j)\}.
$$
Moreover, for $i<n$, we assume that 
$$A_i\leq A_n \ \mathrm{and}\ \mathrm{there}\ \mathrm{exists} \ \mathrm{no} \ [A_i,B_i]\subseteq[A_n,B_n] \ \mathrm{with} \ \zeta_i=\zeta_n.$$

Let $t_n$ be the the smallest nonnegative integer such that $B_n-t_n=B_i$ for some $i<n$ for which $\zeta_i=\zeta_n$. If there is not such an integer, then we set $t_n=\lfloor B_n \rfloor.$

Define $\psi_-$ by
$$
\Jord(\psi_-)=\Jord(\psi)\setminus\{(\rho,A_n,B_n,\zeta_n)\}
$$
and define $(\underline{\zeta}_-,\underline{l}_-, \underline{\eta}_-)$ to be the restriction of $(\underline{\zeta},\underline{l}, \underline{\eta})$ to $\Jord(\psi_-).$

Xu showed that in this above setting the nonvanishing of the representation depends on the nonvanishing of another representations which is defined to be `expansion' of the Jordan block $(\rho,A_n,B_n,\zeta_n).$

\begin{prop}[{\cite[Proposition 7.4]{Xu21}}]\label{prop expand}
    Let $\psi$ and $\psi_-$ be as above. Fix a positive integer $t\leq t_n.$ Then $\pi_>(\psi,\underline{\zeta},\underline{l},\underline{\eta})\neq 0$ if and only if
    \begin{align*}
            \pi_>&(\psi_-,\underline{\zeta}_-,\underline{l}_-,\underline{\eta}_-;(\rho, A_n+t,B_n-t,\zeta_n,l_n+t,\eta_n))\neq 0.
        \end{align*}
\end{prop}
$(\rho, A_n+t,B_n-t,\zeta_n)$ is the expansion of the Jordan block $(\rho,A_n,B_n,\zeta_n).$

We need one more operator called `change sign' which only is defined in the specific cases that $B_1=0,\frac{1}{2}.$
 Let $\psi$ be a local Arthur parameter of good parity, $>_\psi$ be an admissible order, and fix a self-dual irreducible unitary supercuspidal representation $\rho$ of $\GL_d(F).$ Index $\Jord_\rho(\psi)$ such that $(\rho,A_i,B_i,\zeta_i)>(\rho,A_{i-1},B_{i-1},\zeta_{i-1}).$ Suppose there exists $n$ such that for $i>n$, we have
$$
(\rho,A_i,B_i,\zeta_i)\gg_1 \bigcup_{j=1}^n\{(\rho,A_j,B_j,\zeta_j)\}.
$$
Moreover, for $1<i\leq n$, we assume that 
$$A_1\geq A_i \ \mathrm{and}\ \zeta_i\neq \zeta_1.$$
Define $\psi_-$ by
$$
\Jord(\psi_-)=\Jord(\psi)\setminus\{(\rho,A_1,B_1,\zeta_1)\}
$$
and define $(\underline{\zeta}_-,\underline{l}_-, \underline{\eta}_-)$ to be the restriction of $(\underline{\zeta},\underline{l}, \underline{\eta})$ to $\Jord(\psi_-).$

\begin{prop}[{\cite[Propositions 7.5, 7.6]{Xu21}}]\label{prop change sign}
    Let $\psi$ and $\psi_-$ be as above.  
    \begin{enumerate}
        \item If $B_1=0,$ then $\pi_>(\psi,\underline{\zeta},\underline{l},\underline{\eta})\neq 0$ if and only if
    \begin{align*}
            \pi_>&(\psi_-,\underline{\zeta}_-,\underline{l}_-,\underline{\eta}_-;(\rho, A_1,0,-\zeta_1,l_1,\eta_1))\neq 0.
        \end{align*}
        \item Suppose $B_1=\frac{1}{2}.$ If $l_1=\frac{A_1+\frac{1}{2}}{2},$ we set $\eta_1=-1.$
        \begin{enumerate}
            \item If $\eta_1=1$, then $\pi_>(\psi,\underline{\zeta},\underline{l},\underline{\eta})\neq 0$ if and only if
    \begin{align*}
            \pi_>&(\psi_-,\underline{\zeta}_-,\underline{l}_-,\underline{\eta}_-;(\rho, A_1+1,\frac{1}{2},-\zeta_1,l_1+1,-\eta_1))\neq 0.
            \end{align*}
        \item  If $\eta_1=-1$, then $\pi_>(\psi,\underline{\zeta},\underline{l},\underline{\eta})\neq 0$ if and only if
    \begin{align*}
            \pi_>&(\psi_-,\underline{\zeta}_-,\underline{l}_-,\underline{\eta}_-;(\rho, A_1+1,\frac{1}{2},-\zeta_1,l_1,-\eta_1))\neq 0.
            \end{align*}
        \end{enumerate}
    \end{enumerate}
\end{prop}
Hence the change sign operator turns $\zeta_1$ into $-\zeta_1$.

We now state Xu's nonvanishing algorithm. 

\begin{algo}[{\cite[\S8]{Xu21}}]\label{algo Xu nonvanishing} 
    Let $\psi$ be a local Arthur parameter of good parity and fix an admissible order $>$. Set $\Psi=\{\psi\}.$ By performing Steps 1, 2, and 3 repeatedly, we eventually transform $\Psi$ into a set consisting solely of local Arthur parameters in the generalized basic case.
    
{\bf Step 0:} If every $\psi'\in\Psi$ is in the generalized basic case, we determine if $\pi_>(\psi',\underline{\zeta},\underline{l},\underline{\eta})\neq 0$ by Proposition \ref{prop generalized basic case}. If $\pi_>(\psi',\underline{\zeta},\underline{l},\underline{\eta})= 0$, for some $\psi'\in\Psi$, then $\pi_>(\psi,\underline{\zeta},\underline{l},\underline{\eta})= 0$ by Propositions \ref{prop generalized basic case}, \ref{prop pull unequal case}, \ref{prop pull equal case}, \ref{prop expand}, and \ref{prop change sign}.
If $\pi_>(\psi',\underline{\zeta},\underline{l},\underline{\eta})\neq 0$, for every $\psi'\in\Psi$, then $\pi_>(\psi,\underline{\zeta},\underline{l},\underline{\eta})\neq 0$ by Propositions \ref{prop generalized basic case}, \ref{prop pull unequal case}, \ref{prop pull equal case}, \ref{prop expand}, and \ref{prop change sign}. In either case, the algorithm terminates here.

Otherwise, let $\psi'\in\Psi$ be not in the generalized basic case and $\rho$ be a self-dual irreducible unitary supercuspidal representation of $\GL_d(F)$ for which $\Jord_\rho(\psi')$ is not of good shape and proceed to Step 1.
    
{\bf Step 1:} Index $\Jord_\rho(\psi')$ such that $(\rho,A_i,B_i,\zeta_i)>(\rho,A_{i-1},B_{i-1},\zeta_{i-1}).$ Let $n$ be such that for $i>n$, we have
$$
(\rho,A_i,B_i,\zeta_i)\gg_2 \bigcup_{j=1}^n\{(\rho,A_j,B_j,\zeta_j)\},
$$
and that for $i>n$ the Jordan blocks are in good shape. Let $A=\max\{A_i\ | \ i\leq n\}$ and choose a Jordan block $(\rho, A, B, \zeta)\in\cup_{j=1}^n\{(\rho,A_j,B_j,\zeta_j)\}.$ Let 
\begin{equation}\label{eqn S in Xu nonvanish algo}
S=\left\{(\rho, A_i, B_i, \zeta_i)\in\bigcup_{j=1}^n\{(\rho,A_j,B_j,\zeta_j)\} \ | \ [A_i, B_i] \subsetneq [A,B] \ \mathrm{and} \ \zeta_i=\zeta\right\}.    
\end{equation}

If $S\neq\emptyset,$ let $(\rho, A', B', \zeta')\in S$ be such that $A'=\max\{A_i | (\rho, A_i, B_i, \zeta_i)\in S\}.$ Rearrange $>_\psi$ such that $(\rho, A_n, B_n, \zeta_n)=(\rho, A,B,\zeta)$ and $(\rho, A_{n-1}, B_{n-1}, \zeta_{n-1})=(\rho, A',B',\zeta').$ Then we pull $(\rho, A_n, B_n, \zeta_n)$ and $(\rho, A_{n-1}, B_{n-1}, \zeta_{n-1})$ by Proposition \ref{prop pull unequal case} and replace $\psi'$ in $\Psi$ by the 3 local Arthur parameters from Proposition \ref{prop pull unequal case}. We then repeat Step 1.

Suppose that $S=\emptyset$. If there exists another Jordan block $(\rho, A', B', \zeta')\in \Jord_\rho(\psi)$ such that $[A',B']=[A,B]$ and $\zeta'=\zeta,$ then we rearrange $>_\psi$ such that $(\rho, A_n, B_n, \zeta_n)=(\rho, A,B,\zeta)$ and $(\rho, A_{n-1}, B_{n-1}, \zeta_{n-1})=(\rho, A',B',\zeta').$ Then we pull $(\rho, A_n, B_n, \zeta_n)$ and $(\rho, A_{n-1}, B_{n-1}, \zeta_{n-1})$ by Proposition \ref{prop pull equal case} and replace $\psi'$ in $\Psi$ by the 2 local Arthur parameters from Proposition \ref{prop pull equal case}. We then repeat Step 1.

If at some stage $\Jord_\rho(\psi')$ is of good shape, then we return to Step 0. Otherwise, we proceed to Step 2.

{\bf Step 2:} From Step 1, for $\psi'\in\Psi$ we assume that the set
\[\left\{(\rho, A_i, B_i, \zeta_i)\in\bigcup_{j=1}^n\{(\rho,A_j,B_j,\zeta_j)\} \ | \ [A_i, B_i] \subseteq [A,B],  \zeta_i=\zeta\right\}\setminus\{(\rho, A, B, \zeta)\}\]
is empty.
Furthermore, there exists $i\leq n$ such that $(\rho_i, A_i, B_i, \zeta_i)=(\rho, A, B, \zeta)$ since $\Jord_\rho(\psi')$ was not of good shape in Step 1. 
We rearrange $>_\psi$ such that $(\rho, A_n, B_n, \zeta_n)=(\rho, A,B,\zeta)$. We then expand $[A_n,B_n]$ by Proposition \ref{prop expand}. We denote the effect of this expansion on $[A_n,B_n]$ by $[A_n^*,B_n^*]$ and replace $\psi'\in \Psi$ by the resulting expansion. Also note that the resulting local Arthur parameter is still not of good shape. We proceed to Step 3.

{\bf Step 3:} Consider the set
\begin{equation}\label{eqn algo step 3}
    \left\{(\rho, A_i, B_i, \zeta_i)\in\bigcup_{j=1}^{n-1}\{(\rho,A_j,B_j,\zeta_j)\} \ | \ [A_i, B_i] \subsetneq [A_n^*,B_n^*] \ \mathrm{and} \ \zeta_i=\zeta\right\}.
\end{equation}
If this set is nonempty, repeat Step 1.

Otherwise, by definition of `expand,' we must have $B_n^*=0$ or $B_n^*=\frac{1}{2}.$ In either case, we use row exchange (Definition \ref{defn row exchange}) repeatedly to swap $[A_n^*,B_n^*]$ with $[A_i,B_i]$ for $i=1,\dots,n-1.$ Thus $[A_n^*,B_n^*]$ become the first row in the resulting order. We then apply the change sign operator by Proposition \ref{prop change sign} and replace $\psi'$ in $\Psi$ by the resulting local Arthur parameter from Proposition \ref{prop change sign}. We return to Step 1.
\end{algo}

\subsection{Reduction to \texorpdfstring{$\chi_V$}{chi}}

In this subsection, we reduce $\psi$ to the case that the restriction of $\psi$ to $W_F$ is (up to multiplicity) $\chi_V.$ In general, we have that the restriction is a finite sum $\oplus_{i=1}^r m_i\rho_i,$ where $m_i$ is the multiplicity of $\rho_i$ and $\rho_i\neq\rho_j$ for any $i\neq j.$ They key idea is that to determine if a representation given in M{\oe}glin's parameterization does not vanish, then for each $\rho_i$, Xu's nonvanishing algorithm (Algorithm \ref{algo Xu nonvanishing}) gives necessary and sufficient conditions on the parameterization. Importantly, the conditions coming from $\rho_i$ are only affected by the piece of the parameterization attached to $\rho_i.$ This observation leads to the following lemma.

\begin{lemma}\label{lem red to Chi_W}
    Suppose that $\pi\in\Pi_\psi\cap\Pi_{T(\psi)}$ for some $\psi\in\Psi_{gp}(G_n)$ and an operator $T.$ In M{\oe}glin's parameterization, write $\pi=\pi_>(\psi,\underline{\zeta},\underline{l},\underline{\eta})=\pi_{>'}(T(\psi),\underline{\zeta}',\underline{l}',\underline{\eta}').$ Also write $$
    \psi=\bigoplus_{i=1}^r \chi_V\otimes S_{a_i}\otimes S_{b_i} \bigoplus_{\rho\neq\chi_V}\left(\bigoplus_{j\in I_\rho}\rho \otimes S_{a_j}\otimes S_{b_j}\right).$$
    Suppose further that the operator $T$ affects the parameterization on
    $$
\bigoplus_{\rho\neq\chi_V}\left(\bigoplus_{j\in I_\rho}\rho \otimes S_{a_j}\otimes S_{b_j}\right).
    $$
    Let $\pi_\alpha$ and $\pi_{T,\alpha}$ be the corresponding representations described by Recipe \ref{rec pi_alpha}. Let $\alpha\gg 0$ such that $\pi_\alpha=\pi_{T,\alpha}\in\Pi_{\psi_\alpha}\cap\Pi_{T(\psi)_\alpha}$ (such $\alpha$ exists by Theorem \ref{thm Moeglin Adams}). If $\pi_{\alpha-2}\neq 0,$ then $\pi_{T,\alpha-2}\neq 0$.
\end{lemma}

\begin{proof}    
    Algorithm \ref{algo Xu nonvanishing} imposes necessary and sufficient nonvanishing conditions on the  parameterization of $\pi_{T,\alpha-2}$ which determine exactly when $\pi_{T,\alpha-2}\neq 0.$ The parameterizations of $\pi_{T,\alpha-2}$ and $\pi_{T,\alpha}$ agree on
$$\bigoplus_{\rho\neq\chi_V}\left(\bigoplus_{j\in I_\rho}\chi_W\chi_V^{-1}\rho \otimes S_{a_j}\otimes S_{b_j}\right).$$
    Since $\pi_{T,\alpha}\neq 0$, the nonvanishing conditions of Algorithm \ref{algo Xu nonvanishing} must hold. On the other hand, the parameterizations of $\pi_{T,\alpha-2}$ and $\pi_{\alpha-2}$ agree on $\bigoplus_{i=1}^r \chi_W\otimes S_{a_i}\otimes S_{b_i} \oplus \chi_W\otimes S_1\otimes S_{\alpha-2}.$ Hence Algorithm \ref{algo Xu nonvanishing} imposes the same nonvanishing conditions on either of them. Since $\pi_{\alpha-2}\neq 0$, it follows that the parameterization satisfies the nonvanishing conditions. Thus the parameterization of $\pi_{T,\alpha-2}$ satisfies the  nonvanishing conditions of Algorithm \ref{algo Xu nonvanishing} and hence $\pi_{T,\alpha-2}\neq 0.$
\end{proof}

Lemma \ref{lem red to Chi_W} directly admits the following corollary.

\begin{cor}\label{cor rho=chi_v}
    Suppose that $$
    \psi=\bigoplus_{i=1}^r \chi_V\otimes S_{a_i}\otimes S_{b_i} \bigoplus_{\rho\neq\chi_V}\left(\bigoplus_{j\in I_\rho}\rho \otimes S_{a_j}\otimes S_{b_j}\right)$$
    and an operator $T$ affects
    $$
\bigoplus_{\rho\neq\chi_V}\left(\bigoplus_{j\in I_\rho}\rho \otimes S_{a_j}\otimes S_{b_j}\right).
    $$
    If $\pi\in\Pi_\psi\cap\Pi_{T(\psi)}$, then $d(\pi,\psi)=d(\pi,T(\psi)).$ In particular, Theorem \ref{thm main thm} holds in this case.
\end{cor}

\begin{rmk}
    The corollary implies that it is sufficient to prove Theorem \ref{thm main thm} when
    $$
    \psi=\bigoplus_{i=1}^r \chi_V\otimes S_{a_i}\otimes S_{b_i}.
    $$
\end{rmk}

Furthermore, this corollary proves Theorem \ref{thm main thm} for $T=dual_k^-.$ Indeed, $\chi_V$ is an orthogonal representation and $dual_k^-$ affects $\psi$ by changing a summand $\rho\otimes S_{a_k}\otimes S_{a_k+1}$ into $\rho\otimes S_{a_k+1}\otimes S_{a_k}$. However, if $\rho=\chi_V$ neither summand is of good parity since $G$ is symplectic. Thus $\rho\neq\chi_V$. By Corollary \ref{cor rho=chi_v}, we have that Theorem \ref{thm main thm} holds (with equality) if $T=dual_k^-.$ We record this below.

\begin{cor}\label{cor main theorem for dual_k^-}
    Let $\pi$ be a representation of $G_n$ of Arthur type. Suppose that $T=dual_k^-$ is a raising operator and $\pi\in\Pi_\psi\cap\Pi_{T(\psi)}.$ Then
    \begin{equation}
        d^-(\pi,T(\psi))= d^-(\pi,\psi).
    \end{equation}
\end{cor}

\section{Obstructions and proof of Theorem \ref{thm main thm}}\label{sec proof of main thm}

The goal of this section is to prove Theorem \ref{thm main thm}. The outline of the proof is as follows. By Lemma \ref{lem Red of Adams to gp}, it is sufficient to assume that $\pi\in\Pi_\psi\cap\Pi_{T(\psi)}$ for some $\psi\in\Psi_{gp}(G_n)$ and raising operator $T.$ By Theorem \ref{thm Moeglin construction}, M{\oe}glins parameterization gives $\pi=\pi_{>}(\psi,\underline{\zeta},\underline{l},\underline{\eta})$ and $\pi=\pi_{>_T}(T(\psi),\underline{\zeta}_T,\underline{l}_T,\underline{\eta}_T).$ Let $\pi_\alpha$ and $\pi_{T,\alpha}$ be the corresponding representations described by Recipe \ref{rec pi_alpha}. By Theorem \ref{thm Moeglin Adams}, for $\alpha\gg 0$, we have $\pi_\alpha=\pi_{T,\alpha}\in\Pi_{\psi_\alpha}\cap\Pi_{T(\psi)_\alpha}.$ Assume that $\pi_{\alpha-2}\neq 0.$ By Xu's nonvanishing algorithm (Algorithm \ref{algo Xu nonvanishing}), we show that $\pi_{T,\alpha-2}\neq 0$ (Theorem \ref{thm alpha neq 0 implies Talpha neq 0}) by relating the resulting nonvanishing conditions with those of $\pi_{T,\alpha}$ and $\pi_{\alpha-2}$. This implies that $d^-(\pi,T(\psi))\leq d^-(\pi,\psi)$ which proves Theorem \ref{thm main thm}.

Recall that we have proved Theorem \ref{thm main thm} when $T=dual_k^-$ (see Corollary \ref{cor main theorem for dual_k^-}). Thus it remains to consider the cases when $T$ is one of  $ui_{i,j}^{-1}$ or $dual \circ ui_{j,i} \circ dual.$ By \cite[Corollary 5.6]{HLL22}, if $ui_{i,j}^{-1}$ is not of type 3', then $ui_{i,j}^{-1}=dual \circ ui_{j,i} \circ dual.$ Thus, it is enough to consider the case where $T$ is either $ui_{i,j}^{-1}$ of type 3' or $dual \circ ui_{j,i} \circ dual.$

\subsection{Obstructions}

We begin by considering the obstructions introduced by raising operators. These obstructions arise from considering Xu's nonvanishing algorithm (Algorithm \ref{algo Xu nonvanishing}). In particular, each obstruction arises from performing a step of Algorithm \ref{algo Xu nonvanishing} involving the added Jordan block and another certain Jordan block on which a raising operator is applicable. With Theorem \ref{thm main thm} in mind, one should roughly expect that performing the raising operator should remove the obstruction.

The following lemma describes an obstruction to the Adams conjecture when $dual \circ ui_{j,i} \circ dual$ is applicable where $ui_{j,i}$ is not of type 3'.

\begin{lemma}\label{lemma partial obstruction dual ui dual zeta=-1}
    Let $T=dual \circ ui_{j,i} \circ dual$ where $ui_{j,i}$ is not of type 3' and suppose that $\pi\in\Pi_\psi\cap\Pi_{T(\psi)}$ for some $\psi\in\Psi_{gp}(G_n)$. Let $T(\psi)$ be given by replacing the summands 
    \[\chi_V\otimes S_{A_j+\zeta_i B_i+1}\otimes S_{A_j-\zeta_i B_i+1} +\chi_V\otimes S_{A_i+ \zeta_j B_j+1}\otimes S_{A_i- \zeta_j B_j+1} \]
of $\psi$ with
    \[ \chi_V\otimes S_{A_i+\zeta_i B_i+1}\otimes S_{A_i-\zeta_i B_i+1} +\chi_V\otimes S_{A_j+\zeta_j B_j+1}\otimes S_{A_j- \zeta_j B_j+1}.\]
    In M{\oe}glin's parameterization, write $$\pi=\pi_>(\psi,\underline{\zeta},\underline{l},\underline{\eta}).$$  
     Let $\pi_\alpha$ be the corresponding representation described by Recipe \ref{rec pi_alpha}. Let $\alpha\gg 0$ such that $\pi_\alpha\in\Pi_{\psi_\alpha}$ (such $\alpha$ exists by Theorem \ref{thm Moeglin Adams}). If $\frac{\alpha-3}{2}=A_j$ and $\zeta_i=-1,$ then $\pi_{\alpha-2}=0.$
\end{lemma}

\begin{proof} 
Write 
$$
\pi_{\alpha-2}=\pi_{>}(\psi_{\alpha-2},\underline{\zeta}_{\alpha-2},\underline{l}_{\alpha-2},\underline{\eta}_{\alpha-2}).
$$
By \cite[Corollary 5.6]{HLL22}, we have $T=ui_{i,j}^{-1}=dual \circ ui_{j,i} \circ dual.$ Also $T\inv (T(\psi))=\psi.$ Let $(\chi_V, A_i,B_i,\zeta_i),(\chi_V, A_j,B_j,\zeta_j)\in\Jord_{\chi_V}(T(\psi))$ be the Jordan blocks which are affected by $T\inv=ui_{i,j}.$ By Definition \ref{def operators on parameters}, we have that
    \begin{enumerate}
    \item $A_j \geq A_i+1 \geq \zeta_j B_j > -B_i,$ and
        \item  for any $r \in I_{\chi_V}$, if $-B_i<\zeta_r B_r<\zeta_j B_j$, then $A_r \leq A_i $ or $A_r \geq A_j$.
    \end{enumerate}
Note that $ui_{i,j}\inv$ is not of type $3'$ in this case and so $A_j\geq A_i+1>\zeta_j B_j.$ It follows that $A_j>B_j>B_i.$
Furthermore, by Definition \ref{defn row exchange}, we may assume that $j=i+1.$

    We assume that $A_j=\frac{\alpha-3}{2}.$ Using Xu's nonvanishing algorithm, we show that $\pi_{\alpha-2}=0.$ First we suppose that any Jordan block $(\chi_W,A,B,-1)\in\Jord_{\chi_W}(\psi_{\alpha-2})$ with $B>\frac{\alpha-3}{2}$ are far away. Indeed, we ensure this repeatedly using Propositions \ref{prop pull unequal case}(2), \ref{prop pull unequal case}(3), or \ref{prop pull equal case}(2).

We then pull the Jordan blocks
$$\{(\chi_W,\frac{\alpha-3}{2},\frac{\alpha-3}{2},-1),(\chi_W, A_j, B_i, -1)\}\subseteq\Jord_{\chi_W}(\psi_{\alpha-2})$$
by Proposition \ref{prop pull unequal case}(1). Suppose that pulling these blocks satisfied the nonvanishing conditions \eqref{eqn nonvanishing separated} for contradiction. In particular, we have
$$
B_i+l_{\alpha-2}(\chi_W, A_j, B_i, -1)\geq \frac{\alpha-3}{2}=A_j.
$$
Since $A_j>B_j>B_i,$ we have that $A_j-B_i>1.$ Then $l_{\alpha-2}(\chi_W, A_j, B_i, -1)\leq \frac{A_j-B_i+1}{2}< A_j-B_i.$ Thus we obtain a contradiction that $$
A_j> B_i+l_{\alpha-2}(\chi_W, A_j, B_i, -1)\geq A_j.
$$
That is, if $\zeta_i=-1$ and $A_j=\frac{\alpha-3}{2}$, then $\pi_{\alpha-2}=0$.
\end{proof}

Next we consider an obstruction when $T=dual \circ ui_{j,i} \circ dual$ where $ui_{j,i}$ is of type 3'.

\begin{lemma}\label{lemma partial obstruction dual ui dual 3' zeta=-1}
    Let $T=dual \circ ui_{j,i} \circ dual$ where $ui_{j,i}$ is of type 3' and suppose that $\pi\in\Pi_\psi\cap\Pi_{T(\psi)}$ for some $\psi\in\Psi_{gp}(G_n)$. Let $T(\psi)$ be given by replacing the summands 
    \[\chi_V\otimes S_{A_i+\zeta_i B_i+1}\otimes S_{A_i-\zeta_i B_i+1}+\chi_V\otimes S_{A_j+\zeta_j B_j+1}\otimes S_{A_j-\zeta_j B_j+1}\]
of $\psi$ with
    \[ \chi_V\otimes S_{A_i-\zeta_j B_j+1}\otimes S_{A_i-\zeta_j B_j+1}.\]
    In M{\oe}glin's parameterization, write $$\pi=\pi_>(\psi,\underline{\zeta},\underline{l},\underline{\eta}).$$  
     Let $\pi_\alpha$ be the corresponding representation described by Recipe \ref{rec pi_alpha}. Let $\alpha\gg 0$ such that $\pi_\alpha\in\Pi_{\psi_\alpha}$ (such $\alpha$ exists by Theorem \ref{thm Moeglin Adams}). If $\frac{\alpha-3}{2}=A_i$ and $\zeta_i=-1,$ then $\pi_{\alpha-2}=0.$
\end{lemma}

\begin{proof} 
Write 
$$
\pi_{\alpha-2}=\pi_{>}(\psi_{\alpha-2},\underline{\zeta}_{\alpha-2},\underline{l}_{\alpha-2},\underline{\eta}_{\alpha-2}).
$$
By Definition \ref{def operators on parameters}, we have that
    \begin{enumerate}
    \item $A_i \geq A_j+1 = -B_i >-B_j,$ and
        \item  for any $r \in I_{\chi_V}$, if $-B_j<-B_r<-B_i$, then $A_r \leq A_i $ or $A_r \geq A_j$.
    \end{enumerate}
Furthermore, by Definition \ref{defn row exchange}, we may assume that $i=j+1.$

We assume that $A_i=\frac{\alpha-3}{2}$ and $\zeta_i=-1.$ Using Xu's nonvanishing algorithm, we show that $\pi_{\alpha-2}=0.$ First we suppose that any Jordan block $(\chi_W,A,B,-1)\in\Jord_{\chi_W}(\psi_{\alpha-2})$ with $B>\frac{\alpha-3}{2}$ are far away. Indeed, we ensure this repeatedly using Propositions \ref{prop pull unequal case}(2), \ref{prop pull unequal case}(3), or \ref{prop pull equal case}(2).

We then pull the Jordan blocks
$$\{(\chi_W,\frac{\alpha-3}{2},\frac{\alpha-3}{2},-1),(\chi_W, A_i, B_i, -1)\}\subseteq\Jord_{\chi_W}(\psi_{\alpha-2})$$
by Proposition \ref{prop pull unequal case}(1). Suppose that pulling these blocks satisfied the nonvanishing conditions \eqref{eqn nonvanishing separated} for contradiction. Then we can row exchange so that the added Jordan block is indexed by $i+1.$ 
Furthermore, we have
$$
B_i+l_{\alpha-2}(\chi_W, A_i, B_i, -1)\geq \frac{\alpha-3}{2}=A_i
$$
and hence
$$
l_{\alpha-2}(\chi_W, A_i, B_i, -1)\geq A_i-B_i.
$$
However, since $\zeta_i=-1$ we have that $l_{\alpha-2}(\chi_W, A_i, B_i, -1)\leq \frac{A_i-B_i+1}{2}.$ That is,
$$
A_i-B_i \leq l_{\alpha-2}(\chi_W, A_i, B_i, -1)\leq \frac{A_i-B_i+1}{2}
$$
and hence 
$$
0\leq \frac{A_i-B_i}{2}\leq \frac{1}{2}.
$$
There are two possibilities in which these inequalities hold. We have either $A_i-B_i=0$ or $A_i-B_i=1$.

First we consider that $A_i-B_i=0.$ In this case, $b_i=A_i-B_i+1=1$ and so $a_i=A_i+B_i+1\geq b_i.$ Therefore we must have $\zeta_i=1$ which is a contradiction. Hence, in this case $\pi_{\alpha-2}=0.$

Thus, we must have $A_i-B_i=1$ and hence $b_i=A_i-B_i+1=2.$ Since $\zeta_i=-1,$ it follows that $a_i=A_i+B_i+1=1.$ However, the summand $\chi_V\otimes S_1\otimes S_2$ is not of good parity. This contradicts that $T$ is applicable on $\psi$ and hence shows that $\pi_{\alpha-2}=0$ in this case.

Therefore, in any case, we have that $\pi_{\alpha-2}=0.$
\end{proof}

The final obstruction for the case that $T=ui_{i,j}^{-1}$ of type 3' is below.

\begin{lemma}\label{lemma partial obstruction ui inverse 3' zeta=-1}
    Let $T=ui_{i,j}^{-1}$ of type 3' and suppose that $\pi\in\Pi_\psi\cap\Pi_{T(\psi)}$ for some $\psi\in\Psi_{gp}(G_n)$. Let $T(\psi)$ be given by replacing the summand 
    \[\chi_V\otimes S_{A_i+\zeta_i B_i+1}\otimes S_{A_i-\zeta_i B_i+1}\]
of $\psi$ with
    \[ \chi_V\otimes S_{A_i-\zeta_j B_j+1}\otimes S_{A_i-\zeta_j B_j+1},\]
    where $B_j=A_j+1$ and $B_i \leq A_j<A_i.$
    
    In M{\oe}glin's parameterization, write $$\pi=\pi_>(\psi,\underline{\zeta},\underline{l},\underline{\eta}).$$  
     Let $\pi_\alpha$ be the corresponding representation described by Recipe \ref{rec pi_alpha}. Let $\alpha\gg 0$ such that $\pi_\alpha\in\Pi_{\psi_\alpha}$ (such $\alpha$ exists by Theorem \ref{thm Moeglin Adams}). If $\frac{\alpha-3}{2}=A_i$ and $\zeta_i=-1,$ then $\pi_{\alpha-2}=0.$
\end{lemma}

\begin{proof} 
Write 
$$
\pi_{\alpha-2}=\pi_{>}(\psi_{\alpha-2},\underline{\zeta}_{\alpha-2},\underline{l}_{\alpha-2},\underline{\eta}_{\alpha-2}).
$$
We assume that $A_i=\frac{\alpha-3}{2}$ and $\zeta_i=-1.$ Using Xu's nonvanishing algorithm, we show that $\pi_{\alpha-2}=0.$ First we suppose that any Jordan block $(\chi_W,A,B,-1)\in\Jord_{\chi_W}(\psi_{\alpha-2})$ with $B>\frac{\alpha-3}{2}$ are far away. Indeed, we ensure this repeatedly using Propositions \ref{prop pull unequal case}(2), \ref{prop pull unequal case}(3), or \ref{prop pull equal case}(2).

We then pull the Jordan blocks
$$\{(\chi_W,\frac{\alpha-3}{2},\frac{\alpha-3}{2},-1),(\chi_W, A_i, B_i, -1)\}\subseteq\Jord_{\chi_W}(\psi_{\alpha-2})$$
by Proposition \ref{prop pull unequal case}(1) since $B_i<A_i$ (if $B_i=A_i$, then $\zeta_i=1$). Suppose that pulling these blocks satisfied the nonvanishing conditions \eqref{eqn nonvanishing separated} for contradiction. Then we can row exchange so that the added Jordan block is indexed by $i+1.$ 
Furthermore, we have
$$
B_i+l_{\alpha-2}(\chi_W, A_i, B_i, -1)\geq \frac{\alpha-3}{2}=A_i
$$
and hence
$$
l_{\alpha-2}(\chi_W, A_i, B_i, -1)\geq A_i-B_i.
$$
At this point, we obtain an contradiction in exactly the same manner as the proof of Lemma \ref{lemma partial obstruction dual ui dual 3' zeta=-1}. Therefore, we must have that $\pi_{\alpha-2}=0.$
\end{proof}

\subsection{Proof of Theorem \ref{thm main thm}}\label{subsec proof of main thm}

In M{\oe}glin's parameterization, suppose that $$\pi=\pi_>(\psi,\underline{\zeta},\underline{l},\underline{\eta})\neq0.$$ 
Each obstruction in Lemmas \ref{lemma partial obstruction dual ui dual zeta=-1}, \ref{lemma partial obstruction dual ui dual 3' zeta=-1}, and \ref{lemma partial obstruction ui inverse 3' zeta=-1}, required that there exists a Jordan block $(\chi_V,A,B,-1)\in\Jord_{\chi_V}(\psi)$ such that a raising operator $T$ is applicable on $\psi$ which involves the Jordan block $(\chi_V,A,B,-1).$ The following lemma treats the case that the raising operator is applied on Jordan blocks of the form $(\chi_V,A,B,1).$

\begin{lemma}\label{lemma operator on zeta=1}
    Let $T$ be any operator and suppose that $\pi\in\Pi_\psi\cap\Pi_{T(\psi)}$ for some $\psi\in\Psi_{gp}(G_n)$. Furthermore, we suppose that
    \begin{align*}
        &\{(\chi_V,A,B,\zeta)\in\Jord_{\chi_V}(\psi) \ | \ \zeta=-1\}\\
        =&\{(\chi_V,A,B,\zeta)\in\Jord_{\chi_V}(T(\psi)) \ | \ \zeta=-1\}.
    \end{align*}
    In M{\oe}glin's parameterization, write \begin{align*}
        \pi=\pi_>(\psi,\underline{\zeta},\underline{l},\underline{\eta}) 
        =\pi_{>_T}(T(\psi),\underline{\zeta}_T,\underline{l}_T,\underline{\eta}_T).
    \end{align*}  
     Let $\pi_\alpha$ and $\pi_{T,\alpha}$ be the corresponding representation described by Recipe \ref{rec pi_alpha}. Let $\alpha\gg 0$ such that $\pi_\alpha=\pi_{T,\alpha}\neq 0$ (such $\alpha$ exists by Theorem \ref{thm Moeglin Adams}). If $\pi_{\alpha-2}\neq 0$, then $\pi_{T,\alpha-2}\neq 0.$
\end{lemma}

\begin{proof}
By assumption, we have $\pi_\alpha=\pi_{T,\alpha}\in\Pi_{\psi_\alpha}\cap\Pi_{T(\psi)_\alpha}$ and  $\pi_{\alpha-2}\neq 0.$ We show that $\pi_{T,\alpha-2}=\pi_{>_{T,\alpha-2}}(T(\psi)_{\alpha-2},\underline{\zeta}_{\alpha-2},\underline{l}_{\alpha-2},\underline{\eta}_{\alpha-2})\neq 0$ using Algorithm \ref{algo Xu nonvanishing}. Note that the added Jordan block is $(\chi_W, \frac{\alpha-3}{2},\frac{\alpha-3}{2},-1)$. Index $\Jord_{\chi_W}(T(\psi)_{\alpha-2})$ such that $(\rho,A_i,B_i,\zeta_i)>(\rho,A_{i-1},B_{i-1},\zeta_{i-1}).$ Let $n$ be such that for $i>n$, we have
$$
(\chi_W,A_i,B_i,\zeta_i)\gg_2 \bigcup_{j=1}^n\{(\chi_W,A_j,B_j,\zeta_j)\},
$$
and that for $i>n$ the Jordan blocks are in good shape. Let $A=\max\{A_i\ | \ i\leq n\}$ and choose a Jordan block $(\chi_W, A, B, \zeta)\in\cup_{j=1}^n\{(\chi_W,A_j,B_j,\zeta_j)\}.$ Since $T(\psi)_{\alpha}$ and $T(\psi)_{\alpha-2}$ differ only at the added Jordan blocks $(\chi_W, \frac{\alpha-1}{2},\frac{\alpha-1}{2},-1)$ and $(\chi_W, \frac{\alpha-3}{2},\frac{\alpha-3}{2},-1)$, respectively, we focus on the step in the algorithm where the block $(\chi_W, \frac{\alpha-3}{2},\frac{\alpha-3}{2},-1)$ is affected.  We proceed in 4 cases depending on how the added Jordan block is affected by the algorithm.

{\bf Case 1.} Suppose that in Step 1 of Algorithm \ref{algo Xu nonvanishing}, we have  $(\chi_W,\frac{\alpha-3}{2},\frac{\alpha-3}{2},-1)\in S$ (see \eqref{eqn S in Xu nonvanish algo}).
If $A>\frac{\alpha-3}{2}$, since $[\frac{\alpha-3}{2},\frac{\alpha-3}{2}]\subsetneq[A,B]$, we have $[\frac{\alpha-1}{2},\frac{\alpha-1}{2}]\subsetneq[A,B].$ In this case, we pull this pair by Proposition \ref{prop pull unequal case}. In any case of Proposition \ref{prop pull unequal case}, if $B\geq \frac{\alpha-1}{2}$, then the nonvanishing conditions \eqref{eqn nonvanishing separated} for the resulting representation $\pi_{>^*}(T(\psi)_{\alpha-2}^*,\underline{\zeta}^*_{\alpha-2},\underline{l}^*_{\alpha-2},\underline{\eta}^*_{\alpha-2})$ follows from those for the resulting representation $\pi_{>^*}(T(\psi)_{\alpha}^*,\underline{\zeta}^*_{\alpha},\underline{l}^*_{\alpha},\underline{\eta}^*_{\alpha})\neq 0.$ If $B=\frac{\alpha-3}{2},$ the same result is true; however, we note that one needs to apply a row exchange (Definition \ref{defn row exchange}) and use Theorem \ref{thm Moeglin row exchange equiv} to see that the parameterizations of the resulting representations agree. This is needed since the added block in $\Jord_{\chi_W}(T(\psi)_{\alpha-2})$ is affected by row exchange per Recipe \ref{rec pi_alpha} while this would not happen to the added block in $\Jord_{\chi_W}(T(\psi)_{\alpha}).$

Otherwise, we have $A=\frac{\alpha-3}{2}$ and so we may instead choose the Jordan block $(\chi_W, A, B, -1)$ to be the added Jordan block $(\chi_W,\frac{\alpha-3}{2},\frac{\alpha-3}{2},-1)$. The set $S$ in \eqref{eqn S in Xu nonvanish algo} then becomes empty. If $(\chi_W,\frac{\alpha-3}{2},\frac{\alpha-3}{2},-1)$ has multiplicity at least 2 in $\Jord_{\chi_W}(T(\psi)_{\alpha-2})$, then we pull two blocks by Proposition \ref{prop pull equal case}. In this case, we crucially need to use the fact that $\pi_{\alpha-2}\neq0.$ Indeed, per Recipe \ref{rec pi_alpha}, the added Jordan block must be row exchanged with the block $(\chi_W,\frac{\alpha-3}{2},\frac{\alpha-3}{2},-1)$. Suppose that these are indexed by $n$ and $n-1.$ It is only possible to row exchange these blocks if they have the same signs, i.e., $\eta^*_n=\eta^*_{n-1}$ (performing the row exchange in any other case of \ref{defn row exchange} would end up with $l_n'<0$ or $l_{n-1}'<0$ which is not possible). This is guaranteed since $\pi_{\alpha-2}\neq 0$ and the parameterizations of $\pi_{\alpha-2}$ and $\pi_{T,\alpha-2}$ agree on these blocks by assumption. This trick of allowing the added block the be indexed by $n$ is used implicitly below.

If we pull by Proposition \ref{prop pull equal case}(1), then the resulting nonvanishing conditions \eqref{eqn nonvanishing separated} for the pulled pair are equivalent to $l_n=l_{n-1}$. However, since $A=B$, the corresponding summand of $T(\psi)_{\alpha-2}$ is of the form $\chi_W\otimes S_{1}\otimes S_{\alpha-2}$. Since $0\leq l_n,l_{n-1} \leq \frac{\min(1,\alpha-2)}{2}$ are integers, we must have $l_n=l_{n-1}=0$ which gives the nonvanishing conditions \eqref{eqn nonvanishing separated}. The rest of the nonvanishing conditions follow by pulling the added Jordan block in $T(\psi)_\alpha$ away and then pulling the block $(\chi_W,\frac{\alpha-3}{2},\frac{\alpha-3}{2},-1)$ away and using the nonvanishing of $\pi_{T,\alpha}.$

If we pull by Proposition \ref{prop pull equal case}(2), then the nonvanishing conditions \eqref{eqn nonvanishing separated} hold by pulling the added Jordan block of $T(\psi)_\alpha$ away (it is separated already in this case). The resulting representations are exactly the same and hence do not vanish. Therefore, the nonvanishing conditions \eqref{eqn nonvanishing separated} are satisfied if in Step 1 of Algorithm \ref{algo Xu nonvanishing}, we have  $(\chi_W,\frac{\alpha-3}{2},\frac{\alpha-3}{2},-1)\in S$ (see \eqref{eqn S in Xu nonvanish algo}). This completes Case 1.

{\bf Case 2.} Assume there exists $(\chi_W, A, B, -1)\in\cup_{j=1}^n\{(\chi_W,A_j,B_j,\zeta_j)\}$ with $A\geq B>\frac{\alpha-3}{2}$ which is expanded to $(\chi_W, A+t, B-t, -1)$ where $t=B-\frac{\alpha-3}{2}$ in Step 2 of Algorithm \ref{algo Xu nonvanishing}. Note that $A+t>\frac{\alpha-3}{2}$ and so we cannot pull by Proposition \ref{prop pull equal case}. That is, we must pull $(\chi_W, A+t, B-t, -1)$ and $(\chi_W,\frac{\alpha-3}{2},\frac{\alpha-3}{2},-1)$ by Proposition \ref{prop pull unequal case}. We proceed in 3 cases. 

{\bf Case 2(a).} Assume first that we pull by Proposition \ref{prop pull unequal case}(1). Suppose that $\underline{l}_{T,\alpha-2}(\chi_W, A, B, -1)=l$ and note that $\underline{l}_{T,\alpha-2}(\chi_W, \frac{\alpha-3}{2},\frac{\alpha-3}{2}, -1)=0.$ Also we have $\underline{l}_{T,\alpha-2}(\chi_W, A+t, B-t, -1)=l+t.$ The resulting nonvanishing conditions \eqref{eqn nonvanishing separated} are equivalent to $A-l\geq \frac{\alpha-3}{2}$ and $B+l\geq \frac{\alpha-3}{2}$ (or simply the last equality if the signs are unequal). Since $A\geq B>\frac{\alpha-3}{2}$, the Jordan block $(\chi_W, A, B, -1)\in\Jord(T(\psi)_\alpha)$ can be affected by Algorithm \ref{algo Xu nonvanishing} in several ways. Namely, we have either $B>\frac{\alpha-1}{2}$ in which case the block is expanded by Proposition \ref{prop expand}, $A=B=\frac{\alpha-1}{2}$ in which case both blocks are pulled by Proposition \ref{prop pull equal case}, or $A>B=\frac{\alpha-1}{2}$ in which case both blocks are pulled by Proposition \ref{prop pull unequal case}. In each case, we will show that the nonvanishing of $\pi_{T,\alpha}=\pi_{>_{T,\alpha}}(T(\psi)_{\alpha},\underline{\zeta}_{\alpha},\underline{l}_{\alpha},\underline{\eta}_{\alpha})$ implies the nonvanishing conditions mentioned above.

Suppose first that $B>\frac{\alpha-1}{2}$. Then we expand $(\chi_W, A, B, -1)$ to $(\chi_W, A+(t-1), B-(t-1), -1)$ by Proposition \ref{prop expand}. Note that $B-(t-1)=\frac{\alpha-1}{2}.$ Also $\underline{l}_{\alpha}(\chi_W, A, B, -1)=l$ and so $\underline{l}_{\alpha}(\chi_W, A+(t-1), B-(t-1), -1)=l+(t-1).$ Since $\pi_{{>_{T,\alpha}}}(T(\psi)_{\alpha},\underline{\zeta}_{\alpha},\underline{l}_{\alpha},\underline{\eta}_{\alpha})\neq0,$ by Proposition \ref{prop pull unequal case} and the nonvanishing conditions \ref{eqn nonvanishing separated}, we have $A-l\geq \frac{\alpha-1}{2}>\frac{\alpha-3}{2}$ and $B+l\geq \frac{\alpha-1}{2}>\frac{\alpha-3}{2}$ as claimed.  The rest of the nonvanishing conditions follow from those of $\pi_{>_{T,\alpha}}(T(\psi)_{\alpha},\underline{\zeta}_{\alpha},\underline{l}_{\alpha},\underline{\eta}_{\alpha})\neq0.$

Suppose next that $A=B=\frac{\alpha-1}{2}$. In this case, we must have $l=0$ and hence $A-l\geq\frac{\alpha-3}{2}$ and $B+l\geq\frac{\alpha-3}{2}$ and so we are done. Again, the rest of the nonvanishing conditions follow from those of $\pi_{{>_{T,\alpha}}}(T(\psi)_{\alpha},\underline{\zeta}_{\alpha},\underline{l}_{\alpha},\underline{\eta}_{\alpha})\neq0$ (we pull in a similar manner).

Finally, we assume that $A>B=\frac{\alpha-1}{2}$. Since $\pi_{{>_{T,\alpha}}}(T(\psi)_{\alpha},\underline{\zeta}_{\alpha},\underline{l}_{\alpha},\underline{\eta}_{\alpha})\neq0,$ by Proposition \ref{prop pull unequal case}(1) and the nonvanishing conditions \eqref{eqn nonvanishing separated}, we have $A-l\geq \frac{\alpha-1}{2}>\frac{\alpha-3}{2}$ and $B+l\geq \frac{\alpha-1}{2}>\frac{\alpha-3}{2}$ as claimed. Again, the rest of the nonvanishing conditions follow from those of $\pi_{{>_{T,\alpha}}}(T(\psi)_{\alpha},\underline{\zeta}_{\alpha},\underline{l}_{\alpha},\underline{\eta}_{\alpha})\neq0.$

{\bf Case 2(b).} Suppose that we pull $(\chi_W, A+t, B-t, -1)$ and $(\chi_W,\frac{\alpha-3}{2},\frac{\alpha-3}{2},-1)$ by Proposition \ref{prop pull unequal case}(2). In this case, the block $(\chi_W, A+t, B-t, -1)$ is pulled away from the rest. If $B>\frac{\alpha-1}{2}$, then this also happens for $T(\psi)_\alpha$. In this case the algorithm would continue for both $T(\psi)_{\alpha-2}$ and $T(\psi)_\alpha$ into Step 1 again, i.e., we return to considering Case 1 or Case 2 again. 

Thus we may assume that $B=\frac{\alpha-1}{2}.$ In this case, we pull the Jordan block $(\chi_W, A, B, -1)\in\Jord_{\chi_W}(T(\psi)_\alpha)$ by Propositions \ref{prop pull equal case}(2) or \ref{prop pull equal case}(2). Again we return to Case 1 or Case 2.

{\bf Case 2(c).} Suppose that we pull $(\chi_W, A+t, B-t, -1)$ and $(\chi_W,\frac{\alpha-3}{2},\frac{\alpha-3}{2},-1)$ by Proposition \ref{prop pull unequal case}(3). In this case, the resulting representation is nonvanishing since $\pi_{>_T}(T(\psi),\underline{\zeta}_T,\underline{l}_T,\underline{\eta}_T)$ is nonvanishing. Indeed, before arriving to this step, we may perform the corresponding steps in Algorithm \ref{algo Xu nonvanishing} to $\pi_{>_T}(T(\psi),\underline{\zeta}_T,\underline{l}_T,\underline{\eta}_T)$ as we have to $\pi_{T,\alpha-2}=\pi_{>}(T(\psi)_{\alpha-2},\underline{\zeta}_{T,\alpha-2},\underline{l}_{T,\alpha-2},\underline{\eta}_{T,\alpha-2}).$ Note that the indices may have changed since, in Recipe \ref{rec pi_alpha}, we perform the row exchange operator with the added Jordan block. Also, the parameterization of the representations agree on those Jordan blocks $(\cdot, A', B', \zeta')$ with $B'<\frac{\alpha-3}{2}$ (here, $\cdot$ represents $\chi_V$ or $\chi_W$). This remains true after applying the corresponding steps in Algorithm \ref{algo Xu nonvanishing}. Now to pull $(\chi_W, A+t, B-t, -1)$ and $(\chi_W,\frac{\alpha-3}{2},\frac{\alpha-3}{2},-1)$ by Proposition \ref{prop pull unequal case}(3) we must row exchange the added Jordan block to be the last index in $\Jord_{\chi_W}(T(\psi)_{\alpha-2})$. This implies that the parameterization of the representations  $\pi_{>_T}(T(\psi),\underline{\zeta}_T,\underline{l}_T,\underline{\eta}_T)$ and $\pi_{T,\alpha-2}=\pi_{>_{T,\alpha-2}}(T(\psi)_{\alpha-2},\underline{\zeta}_{T,\alpha-2},\underline{l}_{T,\alpha-2},\underline{\eta}_{T,\alpha-2})$ (after applying the previous steps of the algorithm) agree everywhere away from the added Jordan block. Indeed, making the added Jordan block have the last index reverses the row exchanges in Recipe \ref{rec pi_alpha}. Thus, after we pull $(\chi_W, A+t, B-t, -1)$ and $(\chi_W,\frac{\alpha-3}{2},\frac{\alpha-3}{2},-1)$ by Proposition \ref{prop pull unequal case}(3), the nonvanishing conditions of the resulting representation must be satisfied by using Proposition \ref{prop expand} to replace $(\chi_W, A+t, B-t, -1)$ with $(\chi_W, A, B, -1)$ and recognizing that this is the parameterization obtained from $\pi_{>_T}(T(\psi),\underline{\zeta}_T,\underline{l}_T,\underline{\eta}_T)$.

{\bf Case 3.} Assume that there is  $(\chi_W, A, B, 1)\in\cup_{j=1}^n\{(\chi_W,A_j,B_j,\zeta_j)\}$ with $A>\frac{\alpha-3}{2}$ (if $A=\frac{\alpha-3}{2}$ we could focus on the added Jordan block instead) which is expanded to $(\chi_W, A+t, B-t, 1)$ where $B-t\in \{0,\frac{1}{2}\}$ in Step 2 of Algorithm \ref{algo Xu nonvanishing}. Furthermore, we assume that we then apply the change sign operator \ref{prop change sign} in Step 3 of Algorithm \ref{algo Xu nonvanishing} to obtain $(\chi_W, A', B-t, -1)$ and finally pull this Jordan block with the added block $(\chi_W,\frac{\alpha-3}{2},\frac{\alpha-3}{2},-1).$ Note that $A'>\frac{\alpha-3}{2}$ and so we cannot pull by Proposition \ref{prop pull equal case}. That is, we must pull $(\chi_W, A', B-t, 1)$ and $(\chi_W,\frac{\alpha-3}{2},\frac{\alpha-3}{2},-1)$ by Proposition \ref{prop pull unequal case}. Since $A>\frac{\alpha-3}{2}$, we have $A\geq\frac{\alpha-1}{2}$ and so we can also expand the same block $(\chi_W, A, B, 1)\in\Jord_{\chi_W}(T(\psi)_{\alpha})$ to  $(\chi_W, A+t, B-t, 1)$, then change sign to $(\chi_W, A+t, B-t, 1)$, and pull with 
added block $(\chi_W,\frac{\alpha-1}{2},\frac{\alpha-1}{2},-1)$ by Proposition \ref{prop pull equal case}. The rest of this case follows similarly to that of Case 1. 

{\bf Case 4.} We assume now that we must expand the added Jordan block $(\chi_W,\frac{\alpha-3}{2},\frac{\alpha-3}{2},-1)\in\Jord(T(\psi)_{\alpha-2})$ using Proposition \ref{prop expand} in Step 2 of Algorithm \ref{algo Xu nonvanishing}. In this case, we can also expand the added Jordan block $(\chi_W,\frac{\alpha-1}{2},\frac{\alpha-1}{2},-1)\in\Jord(T(\psi)_{\alpha})$ using Proposition \ref{prop expand} in Step 2 of Algorithm \ref{algo Xu nonvanishing}. From Recipe \ref{rec pi_alpha}, we have that the parameterizations of $\pi_{T,\alpha-2}=\pi_{>_{T,\alpha-2}}(T(\psi)_{\alpha-2},\underline{\zeta}'_{\alpha-2},\underline{l}'_{\alpha-2},\underline{\eta}'_{\alpha-2})$
and $\pi_{T,\alpha}=\pi_{>_{T,\alpha}}(T(\psi)_{\alpha},\underline{\zeta}'_{\alpha},\underline{l}'_{\alpha},\underline{\eta}'_{\alpha})$ agree away from the added Jordan blocks in this case. We proceed in three cases.

{\bf Case 4(a).} We suppose that there exists $(\chi_W,A,B,-1)\in\Jord_{\chi_W}(T(\psi)_{\alpha-2})$ such that $A<\frac{\alpha-3}{2}.$ We also assume that $A$ is maximal among such blocks. Then we expand $(\chi_W,\frac{\alpha-3}{2},\frac{\alpha-3}{2},-1)$ to $(\chi_W, \frac{\alpha-3}{2}+t, \frac{\alpha-3}{2}-t,-1)$ by Proposition \ref{prop expand} such that $\frac{\alpha-3}{2}-t=B.$ We then pull the blocks $(\chi_W,A,B,-1)$ and $(\chi_W, \frac{\alpha-3}{2}+t, \frac{\alpha-3}{2}-t,-1)$ by Proposition \ref{prop pull unequal case}. If we pull by Proposition \ref{prop pull unequal case}(1), then the nonvanishing conditions \eqref{eqn nonvanishing separated} for the pulled blocks follow directly since $\frac{\alpha-3}{2}>A$ and $\underline{l}_{T,\alpha-2}(\chi_W,\frac{\alpha-3}{2},\frac{\alpha-3}{2},-1)=0$. The nonvanishing for the resulting representation then follows by performing a similar expansion and pull for the added Jordan block in $\pi_{T,\alpha}\neq 0.$ If we pull by Proposition \ref{prop pull unequal case}(3), then the nonvanishing for the resulting representation again follows by performing a similar expansion and pull for the added Jordan block in $\pi_{T,\alpha}\neq 0.$ If we pull by  \ref{prop pull unequal case}(2), then either there exists another block $(\chi_W,A',B',-1)\in\Jord_{\chi_W}(T(\psi)_{\alpha-2})\setminus\{(\chi_W,A,B,-1)\}$ such that and we repeat Case 4(a) (note that we would pull $(\chi_W,A,B,-1)$ in $\pi_{T,\alpha}$ by \ref{prop pull unequal case}(2) also) or there does not exist such a block and we continue onto the next case.

{\bf Case 4(b).} Suppose there does not exist $(\chi_W,A,B,-1)\in\Jord_{\chi_W}(T(\psi)_{\alpha-2})$ such that $A<\frac{\alpha-3}{2}$; however, we suppose that there exists $(\chi_W,A,B,1)\in\Jord_{\chi_W}(T(\psi)_{\alpha-2})$ such that $A<\frac{\alpha-3}{2}.$ We also assume that $A$ is maximal among such blocks. In this case, we expand the added Jordan block $(\chi_W,\frac{\alpha-3}{2},\frac{\alpha-3}{2},-1)$ to $(\chi_W, \alpha-3, 0,-1)$ by Proposition \ref{prop expand}. We also expand the added Jordan block $(\chi_W,\frac{\alpha-3}{2},\frac{\alpha-3}{2},-1)\in\Jord_{\chi_W}(T(\psi)_\alpha)$ to $(\chi_W, \alpha-1, 0,-1)$ by Proposition \ref{prop expand}. At this point, we apply the change sign operator to both expansions by Proposition \ref{prop change sign} so that the blocks become $(\chi_W, \alpha-3, 0,1)$ and $(\chi_W, \alpha-1, 0,1)$, respectively. Note that the resulting parameterization has $\underline{l}(\chi_W, \alpha-3, 0,1)=\frac{\alpha-3}{2}.$ 

Now we pull the blocks $(\chi_W,A,B,1)$ and $(\chi_W, \alpha-3, 0,1)$ by Proposition \ref{prop pull unequal case}. If we pull by Proposition \ref{prop pull unequal case}(1), then the nonvanishing conditions \eqref{eqn nonvanishing separated} for the pulled blocks follow directly since $\frac{\alpha-3}{2}>A$. The nonvanishing for the resulting representation then follows by performing a similar pull with the added (expanded) Jordan block and using $\pi_{T,\alpha}\neq 0.$ If we pull by Proposition \ref{prop pull unequal case}(3), then the nonvanishing for the resulting representation again follows by performing a similar pull for the added (expanded) Jordan block and using $\pi_{T,\alpha}\neq 0.$ If we pull by \ref{prop pull unequal case}(2) and then if there exists another block $(\chi_W,A',B',1)\in\Jord_{\chi_W}(T(\psi)_{\alpha-2})\setminus\{(\chi_W,A,B,1)\}$,  then we repeat Case 4(b) (note that we would pull $(\chi_W,A,B,1)$ in $\pi_{T,\alpha}$ by \ref{prop pull unequal case}(2) also). 

Otherwise, we have completed Algorithm \ref{algo Xu nonvanishing}. In every step of Algorithm \ref{algo Xu nonvanishing} we obtained nonvanishing representations and therefore, we have that $\pi_{T,\alpha-2}\neq 0$. This completes the proof of the lemma.
\end{proof}

We now prove the main result.

\begin{thm}\label{thm alpha neq 0 implies Talpha neq 0}
    Let $T$ be a raising operator and suppose that $\pi\in\Pi_\psi\cap\Pi_{T(\psi)}$ for some $\psi\in\Psi_{gp}(G_n)$.
    In M{\oe}glin's parameterization, write $$\pi=\pi_>(\psi,\underline{\zeta},\underline{l},\underline{\eta})=\pi_{>_T}(T(\psi),\underline{\zeta}_T,\underline{l}_T,\underline{\eta}_T).$$  Let $\pi_\alpha$ and $\pi_{T,\alpha}$ be the corresponding representations described by Recipe \ref{rec pi_alpha}. Let $\alpha\gg 0$ such that $\pi_\alpha=\pi_{T,\alpha}\in\Pi_{\psi_\alpha}\cap\Pi_{T(\psi)_\alpha}$ (such $\alpha$ exists by Theorem \ref{thm Moeglin Adams}). If $\pi_{\alpha-2}\neq 0,$ then $\pi_{T,\alpha-2}\neq 0$.
\end{thm}

\begin{proof} For simplicity, we prove the theorem in the case that $T=dual\circ ui_{j,i}\circ dual$ where $ui_{j,i}$ is not of type 3'. The cases for the other raising operators follow similarly. Indeed, we sketch the general idea of the proof now. If $T$ only affects Jordan blocks with $\zeta=1$, then we are done by Lemma \ref{lemma operator on zeta=1}. Otherwise, by Lemmas \ref{lemma partial obstruction dual ui dual zeta=-1}, \ref{lemma partial obstruction dual ui dual 3' zeta=-1}, and \ref{lemma partial obstruction ui inverse 3' zeta=-1}, we have that $\frac{\alpha-3}{2}>A$ where $A$ is maximal among the Jordan blocks affected by $T$. This means that the Jordan blocks affected by $T$ are not row exchanged with the added Jordan block and so their parameterizations in both $\pi_{T,\alpha}$ and $\pi_{T,\alpha-2}$ agree. This observation, along with the nonvanishing of $\pi_{\alpha-2}$ and $\pi_{T,\alpha}$, ensures that $\pi_{T,\alpha-2}\neq 0.$ We begin with the details of the proof.

We assume that $T=dual\circ ui_{j,i}\circ dual$ where $ui_{j,i}$ is not of type 3' for simplicity. By \cite[Corollary 5.6]{HLL22}, we have $T=ui_{i,j}^{-1}=dual \circ ui_{j,i} \circ dual.$ Also $T\inv (T(\psi))=\psi.$ Let $(\chi_V, A_i,B_i,\zeta_i),(\chi_V, A_j,B_j,\zeta_j)\in\Jord_{\chi_V}(T(\psi))$ be the Jordan blocks which are affected by $T\inv=ui_{i,j}.$ By Definition \ref{def operators on parameters}, we have that
    \begin{enumerate}
    \item $A_j \geq A_i+1 \geq \zeta_j B_j > \zeta_i B_i,$ and
        \item  for any $r \in I_{\chi_V}$, if $\zeta_i B_i< \zeta_r B_r< \zeta_j B_j$, then $A_r \leq A_i $ or $A_r \geq A_j$.
    \end{enumerate}
Note that $ui_{i,j}\inv$ is not of type $3'$ and so $A_j\geq A_i+1>B_j.$
Furthermore, by Definition \ref{defn row exchange}, we may assume that $j=i+1.$ 
    We have that $T(\psi)$ is given by replacing the summands 
    \[\chi_V\otimes S_{A_j+\zeta_i B_i+1}\otimes S_{A_j-\zeta_i B_i+1} +\chi_V\otimes S_{A_i+\zeta_j B_j+1}\otimes S_{A_i-\zeta_j B_j+1} \]
of $\psi$ with
    \[ \chi_V\otimes S_{A_i+\zeta_i B_i+1}\otimes S_{A_i-\zeta_i B_i+1} +\chi_V\otimes S_{A_j+\zeta_j B_j+1}\otimes S_{A_j-\zeta_j B_j+1}.\]
    Note that the Jordan blocks $(\chi_V, A_i,B_i,\zeta_i),(\chi_V, A_j,B_j,\zeta_j)\in\Jord_{\chi_V}(T(\psi))$ become $(\chi_V, A_j,B_i,\zeta_i),(\chi_V, A_i,B_j,\zeta_j)\in\Jord_{\chi_V}(\psi)$, respectively. Furthermore, the restrictions of $\underline{\zeta},\underline{l},\underline{\eta}$ and $\underline{\zeta}_T,\underline{l}_T,\underline{\eta}_T$ to 
    \begin{align*}
        &\Jord_{\chi_V}(\psi)\setminus\{(\chi_V, A_j,B_i,\zeta_i),(\chi_V, A_i,B_j,\zeta_j)\} \\
        =&\Jord_{\chi_V}(T(\psi))\setminus\{ (\chi_V, A_i,B_i,\zeta_i),(\chi_V, A_j,B_j,\zeta_j)\}
    \end{align*}
    agree. Write $\pi_{\alpha}=\pi_{>}(\psi_{\alpha},\underline{\zeta}_{\alpha},\underline{l}_{\alpha},\underline{\eta}_{\alpha})$ and  $\pi_{T,\alpha}=\pi_{>_{T,\alpha}}(T(\psi)_{\alpha},\underline{\zeta}_{T,\alpha},\underline{l}_{T,\alpha},\underline{\eta}_{T,\alpha})$, and similarly for $\pi_{\alpha-2}$ and $\pi_{T,\alpha-2}.$ We have that the restrictions of $\underline{\zeta}_{\alpha},\underline{l}_{\alpha},\underline{\eta}_{\alpha}$ and $\underline{\zeta}'_{\alpha},\underline{l}_{\alpha}',\underline{\eta}'_{\alpha}$ to 
    \begin{align*}
        &\Jord_{\chi_W}(\psi_{\alpha})\setminus\{(\chi_V, A_j,B_i,\zeta_i),(\chi_V, A_i,B_j,\zeta_j)\} \\
        =&\Jord_{\chi_W}(T(\psi)_{\alpha})\setminus\{ (\chi_W, A_i,B_i,\zeta_i),(\chi_W, A_j,B_j,\zeta_j)\}
    \end{align*}
    agree and similarly for the parameterizations of $\pi_{\alpha-2}$ and $\pi_{T,\alpha-2}.$

    We proceed by considering Xu's nonvanishing algorithm (Algorithm \ref{algo Xu nonvanishing}) for $\pi_{T,\alpha-2}$. The goal is to show that, in each step, the resulting nonvanishing conditions follow by considering how Xu's nonvanishing algorithm affects $\pi_{\alpha-2}$ and $\pi_{T,\alpha-2}.$ If $\zeta_i=1$, then $\zeta_j=1$ and we are done by Lemma \ref{lemma operator on zeta=1}.
    
    Therefore, we suppose that $\zeta_i=-1.$ By Lemma \ref{lemma partial obstruction dual ui dual zeta=-1} (or Lemmas \ref{lemma partial obstruction dual ui dual 3' zeta=-1} or \ref{lemma partial obstruction ui inverse 3' zeta=-1} for the other raising operators), we have that $\frac{\alpha-3}{2}>A_j.$ Note that the Algorithm \ref{algo Xu nonvanishing} for $\pi_{T,\alpha-2}$ follows exactly as it does for $\pi_{\alpha-2}$ and $\pi_{T,\alpha}$ except when considering the added Jordan block or the blocks affected by $T$, i.e., the Jordan blocks
    $$(\chi_W,A_i,B_i,\zeta_i),(\chi_W,A_j,B_j,\zeta_j),(\chi_W,\frac{\alpha-3}{2},\frac{\alpha-3}{2},-1)\in\Jord_{\chi_W}(T(\psi)_{\alpha-2}).$$ Thus we assume that $\pi_{T,\alpha-2}$ and $\pi_{\alpha-2}$ have no Jordan blocks $(\chi_W,A,B,\zeta)$ where $B>\frac{\alpha-3}{2}.$ Since $\frac{\alpha-3}{2}>A_j>A_i$, the added Jordan block must be encountered first in Algorithm \ref{algo Xu nonvanishing} and we proceed in cases based on the steps of Algorithm \ref{algo Xu nonvanishing}.
    
    {\bf Case 1.} Suppose the added block $(\chi_W,\frac{\alpha-3}{2},\frac{\alpha-3}{2},-1)\in\Jord_{\chi_W}(T(\psi)_{\alpha-2})$ is pulled by Propositions \ref{prop pull unequal case} or \ref{prop pull equal case} (this is a case in Step 1 of Algorithm \ref{algo Xu nonvanishing}). If the pull shifts only the block $(\chi_W,\frac{\alpha-3}{2},\frac{\alpha-3}{2},-1)$, then the resulting parameterization (with appropriate restrictions) is exactly that of $\pi_{>'}(T(\psi),\underline{\zeta}_T,\underline{l}_T,\underline{\eta}_T)$ from which the nonvanishing conditions for the rest of the algorithm follow. If the pull only shifts another block, then we simply continue the algorithm into another case.
        
        If the pull shifts $(\chi_W,\frac{\alpha-3}{2},\frac{\alpha-3}{2},-1)\in\Jord_{\chi_W}(T(\psi)_{\alpha-2})$ with another Jordan block, since $\frac{\alpha-3}{2}>A_j$, the block must not be either of the blocks affected by $T$. Therefore, the added block must be shifted with $(\chi_W,A', B', -1)\in\Jord(T(\psi)_{\alpha-2})$ where $A'\geq \frac{\alpha-3}{2}\geq B'.$
        
        We assume that $(\chi_W,A', B', -1)=(\chi_W, A_n,B_n,-1)$ and $(\chi_W,\frac{\alpha-3}{2},\frac{\alpha-3}{2},-1)=(\chi_W,A_{n-1}, B_{n-1},-1).$ Indeed, this would be the case in Algorithm \ref{algo Xu nonvanishing}. Furthermore, by Definitions \ref{def operators on parameters} and \ref{defn row exchange}, it may be assumed that in the ordering on $\Jord_{\chi_V}(T(\psi))$ that $j=i+1<n$, i.e., we did the row exchange before lifting. Hence applying Algorithm \ref{algo Xu nonvanishing} to $\pi_{\alpha-2}$, we also pull $(\chi_W,A', B', -1)=(\chi_W, A_n,B_n,-1)$ and $(\chi_W,\frac{\alpha-3}{2},\frac{\alpha-3}{2},-1)=(\chi_W,A_{n-1}, B_{n-1},-1).$ 
        
        Let $\pi_{\alpha-2}'$ be the resulting representation obtained after pulling $(\chi_W,A', B', -1)$ and $(\chi_W,\frac{\alpha-3}{2},\frac{\alpha-3}{2},-1)$ away, by Proposition \ref{prop pull unequal case}(1) or \ref{prop pull equal case}(1), and then restricting $\underline{\zeta}_{\alpha-2},\underline{l}_{\alpha-2},$ and $\underline{\eta}_{\alpha-2}$ appropriately. Since $\pi_{\alpha-2}\neq 0$, we have $\pi_{\alpha-2}'\neq 0$ and the pulled blocks $$\{(\chi_W,A', B', -1),(\chi_W,\frac{\alpha-3}{2},\frac{\alpha-3}{2},-1)\}$$
        satisfy the nonvanishing conditions \eqref{eqn nonvanishing separated}. Also $T$ is applicable on $\pi_>(\psi,\underline{\zeta},\underline{l},\underline{\eta})$ and the restriction of $\underline{\zeta}_{\alpha-2},\underline{l}_{\alpha-2},$ and $\underline{\eta}_{\alpha-2}$ agrees with $\underline{\zeta},\underline{l}$ and $\underline{\eta}$ at the $i,j$-th Jordan blocks, i.e, those which $T$ affects. It follows that $T$ is applicable on $\pi_{\alpha-2}'.$ Let $\pi_{T,\alpha-2}'$ be defined analogously from $\pi_{T,\alpha-2}$ as $\pi_{\alpha-2}'$ is defined from $\pi_{\alpha-2}.$  Then $\pi_{T,\alpha-2}'=\pi_{\alpha-2}'\neq 0$ by applying $T.$ Also, the pulled blocks $$\{(\chi_W,A', B', -1),(\chi_W,\frac{\alpha-3}{2},\frac{\alpha-3}{2},-1)\}\subseteq\Jord_{\chi_W}(T(\psi)_{\alpha-2})$$ satisfy the nonvanishing conditions \eqref{eqn nonvanishing separated} as noted above (the parameterizations of $\pi_{\alpha-2}$ and $\pi_{T,\alpha-2}$ agree on these blocks).

        Therefore, if $\frac{\alpha-3}{2}>A_j$ and the block $(\chi_W,\frac{\alpha-3}{2},\frac{\alpha-3}{2},-1)\in\Jord_{\chi_W}(T(\psi)_{\alpha-2})$ is pulled away in Algorithm \ref{algo Xu nonvanishing} by Propositions \ref{prop pull unequal case} or \ref{prop pull equal case}, then the resulting representations are nonvanishing. This completes Case 1.

        {\bf Case 2.} We assume that added Jordan block is involved in Step 2 of Algorithm \ref{algo Xu nonvanishing}. That is, we must expand the added Jordan block  $(\chi_W,\frac{\alpha-3}{2},\frac{\alpha-3}{2},-1)=(\chi_W,A,B,-1)$ by Proposition \ref{prop expand}. Furthermore, for any Jordan block $$(\chi_W,A',B',\zeta)\in\Jord_{\chi_W}(T(\psi)_{\alpha-2})\setminus\{(\chi_W,\frac{\alpha-3}{2},\frac{\alpha-3}{2},-1)\}$$ we have $\frac{\alpha-3}{2}> A'$ (otherwise we are back into Case 1. Let $(\chi_W,A',B',-1)$ be a Jordan block in the above set such that $A'$ is maximal. Let $t\in\mathbb{Z}$ be a positive integer such that $\frac{\alpha-3}{2}-t=B'.$ It follows that $t\geq A'-B'+1.$ Let $l'$ denote the value of $\underline{l}_{T,\alpha-2}$ on the Jordan block
$(\chi_W,A',B',-1)$. Then \[0\leq l'\leq \frac{\min(A'+B'+1, A'-B'+1)}{2}\leq t.\] By Proposition \ref{prop expand}, we expand $(\chi_W,\frac{\alpha-3}{2},\frac{\alpha-3}{2},-1)$ to  $(\chi_W,\frac{\alpha-3}{2}+t,B',-1).$ Note that the value of $\underline{l}_{T,\alpha-2}$ on the Jordan block
$(\chi_W,A,B,-1)$ was $0$ and on the expanded Jordan block $(\chi_W,\frac{\alpha-3}{2}+t,B',-1)$ it becomes $t.$ Algorithm \ref{algo Xu nonvanishing} then proceeds to Step 1 and pulls both $(\chi_W,\frac{\alpha-3}{2}+t,B',-1)$ and $(\chi_W,A',B',-1)$ by Proposition \ref{prop pull unequal case} (we cannot pull by Proposition \ref{prop pull equal case} since $\frac{\alpha-3}{2}+t>A'$). 

If we pull both blocks by Proposition \ref{prop pull unequal case}(1), then we see that the nonvanishing conditions \eqref{eqn nonvanishing separated} become
$\frac{\alpha-3}{2}\geq A'-l'$ and $t\geq l'$ or just $B'+t=\frac{\alpha-3}{2}\geq A'-l'$. In either case, these conditions hold by the above discussion. Let $\pi_{T,\alpha-2}'$ be the resulting representation obtained after pulling $(\chi_W,A', B', -1)$ and $(\chi_W,\frac{\alpha-3}{2},\frac{\alpha-3}{2},-1)$ away, by Proposition \ref{prop pull unequal case}(1) and then restricting $\underline{\zeta}_{T,\alpha-2},\underline{l}_{T,\alpha-2},$ and $\underline{\eta}_{T,\alpha-2}$ appropriately. In this case, we can also similarly pull the Jordan block $(\chi_W,\frac{\alpha-1}{2},\frac{\alpha-1}{2},-1)\in\Jord_{\chi_W}(T(\psi)_\alpha)$  with $(\chi_W,A', B', -1)$. Let $\pi_{T,\alpha}'$ denote resulting representation obtained after pulling both $(\chi_W,A', B', -1)$ and $(\chi_W,\frac{\alpha-1}{2},\frac{\alpha-1}{2},-1)$ away, by Proposition \ref{prop pull unequal case}(1) and then restricting $\underline{\zeta}_{T,\alpha},\underline{l}_{T,\alpha},$ and $\underline{\eta}_{T,\alpha}$ appropriately. It follows that after restricting, we have $\underline{\zeta}_{T,\alpha}=\underline{\zeta}_{T,\alpha-2},\underline{l}_{T,\alpha}=\underline{l}_{T,\alpha-2},$ and $\underline{\eta}_{T,\alpha}=\underline{\eta}_{T,\alpha}.$ Thus, since $\pi_{T,\alpha}\neq 0$, we have that $\pi_{T,\alpha-2}'=\pi_{T,\alpha}'\neq0.$

If instead we pull these blocks by Proposition \ref{prop pull unequal case}(2) or (3), then the nonvanishing conditions \eqref{eqn nonvanishing separated} follows similarly from those for $\pi_{T,\alpha}\neq 0.$ 

Therefore, if we expand the added Jordan block and then pull by Proposition \ref{prop pull unequal case}, then the resulting representations are nonvanishing. This completes Case 2.

{\bf Case 3.} So far, we have considered the effects of Algorithm \ref{algo Xu nonvanishing} on the added Jordan block if it is involved in Steps 1 or 2. If the added Jordan block is involved in Step 3, then we would expand $(\chi_W,\frac{\alpha-3}{2},\frac{\alpha-3}{2},-1)$ to $(\chi_W,{\alpha-3},0,-1)$. To apply Proposition \ref{prop change sign}, we must have that the set \eqref{eqn algo step 3} is empty. However, this set cannot be empty since $\frac{\alpha-3}{2}>A_j>B_j>B_i.$ In particular, $B_j\geq 1$ and so the furthest we can expand $(\chi_W,\frac{\alpha-3}{2},\frac{\alpha-3}{2},-1)$ is to $$(\chi_W,\frac{\alpha-3}{2}+t,\frac{\alpha-3}{2}-t,-1)$$ where $\frac{\alpha-3}{2}-t\geq B_i.$ Thus the added Jordan block cannot be involved in Step 3 of Algorithm \ref{algo Xu nonvanishing}. This completes Case 3.

Therefore, we must always be in Cases 1 or 2. Hence, using Algorithm \ref{algo Xu nonvanishing}, we have that if $\zeta_i=-1,$ then $\pi_{T,\alpha-2}\neq0.$ 
This completes the proof of the theorem.
\end{proof}

Next we discuss several examples which illustrate the ideas of this paper.

\begin{exmp}\label{exmp stable Arthur} We consider the example in \cite[Example 7.3]{BH22}.
    Recall that $\chi_V$ is trivial. Consider the following local Arthur parameters of good parity of $\Sp_{10}(F)$.
    \begin{align*}
        \psi_1&=\chi_V\otimes S_1\otimes S_7+\chi_V\otimes S_1\otimes S_3 + \chi_V\otimes S_1\otimes S_1, \\
        \psi_2=T_1(\psi_1)&=\chi_V\otimes S_1\otimes S_7+\chi_V\otimes S_2\otimes S_2 \\
        \psi_3=T_2(\psi_2)&=\chi_V\otimes S_1\otimes S_7+\chi_V\otimes S_3\otimes S_1 +\chi_V\otimes S_1\otimes S_1,
    \end{align*}
    where $T_1=dual\circ ui_{3,2}\circ dual$ and $T_2=ui\inv_{2}.$ Let $\pi$ be the unique irreducible quotient of the standard module $|\cdot|^3\rtimes\pi_0$ where $\pi_0$ is the unique supercuspidal representation in the L-packet associated to the $L$-parameter $\phi=1_{W_F}\otimes S_1+1_{W_F}\otimes S_3+1_{W_F}\otimes S_5$. It follows that the $L$-parameter of $\pi$ is
    $$
    \phi_\pi=|\cdot|^3\otimes S_1 + |\cdot|^{-3}\otimes S_1+1_{W_F}\otimes S_1+1_{W_F}\otimes S_3+1_{W_F}\otimes S_5.
    $$
    In particular, $\phi_\pi$ is not of Arthur type. However, we have that $$
    \pi\in\Pi_{\psi_1}\cap\Pi_{\psi_2}\cap\Pi_{\psi_3}
    $$
    and by \cite[Theorem 1.4]{HLL22}, these are all the local Arthur packets to which $\pi$ belongs. Furthermore, $\psi_1=\psi^{min}(\pi)$ and $\psi_3=\psi^{max}(\pi).$ For $j=1,2,3$ and $\alpha$ a positive odd integer, let
    $$
    (\psi_j)_\alpha=\chi_W\chi_V\inv \psi + \chi_W\otimes S_1\otimes S_\alpha.
    $$
    It is computed in \cite[Example 7.3]{BH22} that on the going up tower, the first occurrence of $\pi$ is when $\alpha=9.$ From the conservation relations (Theorem \ref{thm conservation relation}), on the going down tower the first occurrence is when $\alpha=-7.$ However, the Adams conjecture requires $\alpha>0.$ From \cite[Example 7.3]{BH22}, $\theta_{-5}(\pi)$ is the first place where Adams conjecture holds. By \cite[Theorem C]{BH22} and Theorem \ref{thm main thm}, we have $\theta_{-\alpha}(\pi)\in\Pi_{(\psi_j)_\alpha}$ for any $\alpha\geq 5.$ In particular, $d(\pi,\psi^{max}(\pi))=5.$ Conjecture \ref{conj stable Arthur} implies that $m_A^{-,\alpha}(\pi)=5.$ We confirm this below.

    Combining the results of \cite{AG17a,BH21}, we obtain that $\theta_{-1}(\pi), \theta_{-3}(\pi),$ and $\theta_{-5}(\pi)$ are the unique irreducible quotients of the standard modules 
    \begin{align*}
        |\cdot|^3\rtimes\theta_{-1}&(\pi_0), \\
        |\cdot|^3\times |\cdot|\rtimes\theta_{-1}&(\pi_0), \\
        |\cdot|^3\times|\cdot|^2\times|\cdot|\rtimes\theta_{-1}&(\pi_0),
    \end{align*}
    respectively. Note that $\theta_{-1}(\pi_0)$ is a tempered representation whose $L$-parameter is
    $$
    \phi_{\theta_{-1}(\pi_0)}=1_{W_F}\otimes S_1+1_{W_F}\otimes S_1+1_{W_F}\otimes S_3+1_{W_F}\otimes S_5.
    $$ 
    Thus the $L$-parameters of $\theta_{-1}(\pi)$ and $\theta_{-3}(\pi)$ are given by
    \begin{align*}
        \phi_{\theta_{-1}(\pi)}&=|\cdot|^3\otimes S_1 + |\cdot|^{-3}\otimes S_1\\
        &+1_{W_F}\otimes S_1+1_{W_F}\otimes S_1+1_{W_F}\otimes S_3+1_{W_F}\otimes S_5, \\
        \phi_{\theta_{-3}(\pi)}&=|\cdot|^3\otimes S_1 + |\cdot|^{1}\otimes S_1+|\cdot|^{-1}\otimes S_1 + |\cdot|^{-3}\otimes S_1 \\
        &+1_{W_F}\otimes S_1+1_{W_F}\otimes S_1+1_{W_F}\otimes S_3+1_{W_F}\otimes S_5.
    \end{align*}
    Note that neither of these $L$-parameters are of Arthur type. Furthermore, neither $\theta_{-1}(\pi)$ nor $\theta_{-3}(\pi)$ are of Arthur type. Indeed, here is an outline of the argument. Suppose one of them is of Arthur type for contradiction. Then we apply the local theta correspondence to lift it to a representation of $\Sp_{n'}(F)$ for some positive even integer $n'$ large enough such that the Adams conjecture holds (\cite[Theorem 5.1]{Moe11c}), i.e., its lift is of Arthur type. We compute explicitly its Langlands classification using \cite{AG17a, BH21}. We then use \cite[Algorithm 7.9]{HLL22} to see that this representation is not of Arthur type which gives a contradiction. 
    
    We remark that it is not surprising that neither $\theta_{-1}(\pi)$ nor $\theta_{-3}(\pi)$ are of Arthur type. Indeed, Conjecture \ref{conj stable Arthur} predicts that $\theta_{-3}(\pi)$ is not of Arthur type. Thus $m_A^{-,\alpha}(\pi)=5=d^-(\pi,\psi^{max}(\pi))$ which agrees with Conjecture \ref{conj stable Arthur}.
    
    Also, it follows from \cite{AG17a,BH21} that the first occurrence of $\pi$, $\theta_7(\pi)$, is tempered with $L$-parameter given by
    \[
     \phi_{\theta_{7}(\pi)}=1_{W_F}\otimes S_1 +1_{W_F}\otimes S_3.
    \]
    In particular, $\theta_7(\pi)$ is of Arthur type. That is, it is possible for $\theta_{-\alpha}(\pi)$ to be of Arthur type for $\alpha<m_A^{-,\alpha}(\pi).$
\end{exmp}

\begin{exmp}\label{exmp supercuspidal}
    Let  $\psi_1=\chi_V \otimes S_3 \otimes S_3,$ and $\pi\in\Pi_\psi$ be the unique supercuspidal representation. We have that $|\Psi(\pi)|=9.$ We let $D$ denote an operator of the form $dual\circ ui\circ dual$ for brevity. The $9$ local Arthur parameters in $\Psi(\pi)$ are shown below along with a raising operator relating them:
   $$
   
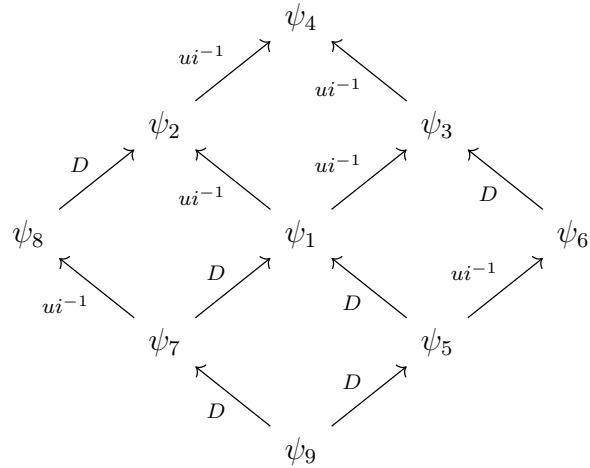
\begin{tikzcd}[ampersand replacement=\&]
   \& \& \psi_4 \& \& \\
   \& \psi_2 \arrow[ur, "ui\inv"] \&  \& \psi_3 \arrow[ul, "ui\inv"] \& \\
   \psi_8 \arrow[ur, "D"] \& \& \psi_1 \arrow[ul, "ui\inv"] \arrow[ur, "ui\inv"] \& \& \psi_6 \arrow[ul, "D"] \\
   \& \psi_7 \arrow[ur, "D"] \arrow[ul, "ui\inv"] \& \& \psi_5 \arrow[ul, "D"] \arrow[ur, "ui\inv"] \\
   \& \& \psi_9 \arrow[ul , "D"] \arrow[ur, "D"] \& \&
   \end{tikzcd}
$$
where
\begin{align*}
    \psi_1&= \chi_V \otimes S_3 \otimes S_3,\\
    \psi_2&= \chi_V \otimes S_1 \otimes S_1 + \chi_V \otimes S_4 \otimes S_2,\\
    \psi_3&= \chi_V \otimes S_2 \otimes S_2 + \chi_V \otimes S_5 \otimes S_1,\\
    \psi_4&= \chi_V \otimes S_1 \otimes S_1 + \chi_V \otimes S_3 \otimes S_1 +\chi_V\otimes S_5\otimes S_1,\\
    \psi_5&= \chi_V \otimes S_1 \otimes S_1 + \chi_V \otimes S_2 \otimes S_4,\\
    \psi_6&= \chi_V \otimes S_1 \otimes S_1 + \chi_V \otimes S_1 \otimes S_3 +\chi_V\otimes S_5\otimes S_1,\\
    \psi_7&= \chi_V \otimes S_1 \otimes S_5 + \chi_V \otimes S_2 \otimes S_2,\\
    \psi_8&= \chi_V \otimes S_1 \otimes S_1 + \chi_V \otimes S_1 \otimes S_5 +\chi_V\otimes S_3\otimes S_1,\\
    \psi_9&= \chi_V \otimes S_1 \otimes S_1 + \chi_V \otimes S_1 \otimes S_3 +\chi_V\otimes S_1\otimes S_5.
\end{align*}
Note that by Definition \ref{def operator ordering}, we understand $\geq_O$ on $\Psi(\pi).$ Therefore, by Theorem \ref{thm psi max min}, we have $\psi^{max}(\pi)=\psi_4$ and $\psi^{min}(\pi)=\psi_9.$

Next we consider the theta lift $\theta_{-\alpha}(\pi).$ Note that $n=10.$ The first occurrence is when  $m^-(\pi)=4$ and hence by Theorem \ref{thm conservation relation}, $m^+(\pi)=20.$
Lifting to either $V_{20}^\pm$ has $\theta_{-9}^\pm(\pi)\in\Pi_{{\psi_i}_9}$ for any $i=1,\dots,9.$ Lifting to $V_{18}^+$ has $\theta_{-7}^+(\pi)=0$ by Theorem \ref{thm BH up down}. Also by Theorem \ref{thm BH up down}, we have that lifting to $V_{18}^-$ has $\theta_{-7}^-(\pi)\in\Pi_{(\psi_i)_9}$ for any $i.$ However, lifting to $V_{16}^-$ depends on $\psi_i$. Indeed, obstructions are introduced by Lemmas \ref{lemma partial obstruction dual ui dual zeta=-1}, \ref{lemma partial obstruction dual ui dual 3' zeta=-1}, and \ref{lemma partial obstruction ui inverse 3' zeta=-1}. 

Consider the case $\psi=\psi_9$. We have that $A_2=2$ and we can apply the operator $T=dual \circ ui_{2,1} \circ dual$ where $ui_{2,1}$ is not of type 3'.  Lemma \ref{lemma partial obstruction dual ui dual zeta=-1} implies that $\pi_{5}=0.$ Consequently, $\theta^-_{-5}(\pi)\not\in\Pi_{(\psi_9)_5}$ and $d(\pi,\psi_9)=7.$

On the other hand, one can check directly that $\theta^-_{-\alpha}(\pi)\in\Pi_{(\psi_4)_\alpha}$ for any $\alpha.$ In particular,
for $\alpha=1,$ we have $\theta^-_{-1}(\pi)\in\Pi_{({\psi_4})_1}$ where
$$
{({\psi_4})_1}= \chi_W \otimes S_1 \otimes S_1+\chi_W \otimes S_1 \otimes S_1 + \chi_W \otimes S_3 \otimes S_1 +\chi_W\otimes S_5\otimes S_1
$$
is tempered. Note that $\theta^-_{1}(\pi)$ is also tempered (its $L$-parameter is of the form $\chi_W \otimes S_3+\chi_W\otimes S_5$) and hence of Arthur type. Thus $m^{-,\alpha}_A(\pi)<0.$ However, $d^-(\pi,\psi^{max}(\pi))=d^-(\pi,\psi_4)=1$ and so Conjecture \ref{conj stable Arthur} is not applicable.
\end{exmp}

\bibliographystyle{amsplain}
\bibliography{Adams_arxiv}

\end{document}